\documentclass[twocolumn,amsthm]{autart}

\usepackage{graphics} % for pdf, bitmapped graphics files
\usepackage{epstopdf}
\usepackage{amsmath,amssymb,nccmath} % assumes amsmath package installed

\newtheorem{theorem}{Theorem}[section]
\newtheorem{corollary}[theorem]{Corollary}
\newtheorem{lemma}[theorem]{Lemma}
\newtheorem{remark}[theorem]{Remark}
\newtheorem{problem}[theorem]{Problem}

\newtheorem{definition}[theorem]{Definition}

\usepackage{multicol}
\usepackage{natbib}
\usepackage{pgf-pie}
\usepackage{pgfplots}
\usepackage{booktabs}
\usepackage{verbatim}
\usepackage{subcaption}
\usepackage{graphicx} % support the \includegraphics command and options
\usepackage{multirow}
\usepackage{lipsum}

\makeatletter
\def\footnoterule{\relax%
\kern-5pt
\hbox to \columnwidth{\hfill\vrule width .9\columnwidth height 0.4pt\hfill}
\kern4.6pt}
\makeatother

\usepackage[bb=dsserif]{mathalpha}
\usepackage{bm}

\usepackage{color,hyperref}
\definecolor{darkblue}{rgb}{0.0,0.0,0.6}
\hypersetup{
colorlinks=true,       % false: boxed links; true: colored links
linkcolor=darkblue,    % color of internal links (change box color with linkbordercolor)
citecolor=darkblue,    % color of links to bibliography
filecolor=darkblue,    % color of file links
urlcolor=darkblue      % color of external links
}

\begin{document}

\begin{frontmatter}
%\runtitle{Insert a suggested running title}  % Running title for regular papers but only if the title is over 5 words. The running title is not shown in the output.

\title{On Cost-Sensitive Distributionally \\ Robust Log-Optimal Portfolio} % Title, preferably not more than 10 words.

\thanks[footnoteinfo]{This paper was supported in part by National Science and Technology Council (NSTC), Taiwan, under Grant: NSTC113--2628--E--007--015--.}

\author[CHHSIEH]{Chung-Han Hsieh}\ead{ch.hsieh@mx.nthu.edu.tw. Corresponding author.} \; and \;
\author[CHHSIEH]{Xiao-Rou Yu} \ead{zoe890324@gmail.com}
%~$\;$ and $\;$    % Add the 
%\author[CHHSIEH]{B. Ross Barmish}\ead{barmish@engr.wisc.edu}               % e-mail address 
%\author[Baiae]{Publius Maro Vergilius}\ead{vergilius@culture.ir}  % (ead) as shown

\address[CHHSIEH]{Department of Quantitative Finance, National Tsing Hua University, Hsinchu, Taiwan, 300044.}
%\address[Barmish]{Senate House, Rome}             % full addresses
%\address[Baiae]{The White House, Baiae}        % here.

\begin{keyword}  % Five to ten keywords,  
	distributional robust optimization, stochastic financial systems, robustness, risk management.               % chosen from the IFAC 
\end{keyword}                             
% keyword list or with the help of the Automatica keyword wizard

\begin{abstract}     
	This paper addresses a novel \emph{cost-sensitive} distributionally robust log-optimal portfolio problem, where the investor faces \emph{ambiguous} return distributions, and a general convex transaction cost model is incorporated.  The uncertainty in the return distribution is quantified using the \emph{Wasserstein} metric, which captures distributional ambiguity. We establish conditions that ensure robustly survivable trades for all distributions in the Wasserstein ball under convex transaction costs. By leveraging duality theory, we approximate the infinite-dimensional distributionally robust optimization problem with a finite convex program, enabling computational tractability for mid-sized portfolios. Empirical studies using S\&P 500 data validate our theoretical framework: without transaction costs, the optimal portfolio converges to an equal-weighted allocation, while with transaction costs, the portfolio shifts slightly towards the risk-free asset, reflecting the trade-off between cost considerations and optimal allocation.	
 \end{abstract}

\end{frontmatter}

\section{Introduction} \label{section: Introduction}
This paper presents a novel approach to financial portfolio optimization, framed within a control-theoretic perspective. Specifically, we formulate a cost-sensitive, distributionally robust log-optimal portfolio optimization problem that integrates the Wasserstein metric for ambiguity consideration, alongside a general convex transaction cost model. This framework bridges the gap between financial portfolio management and optimal control, emphasizing robustness and computational tractability in real-world trading.

Portfolio management can be viewed as a stochastic optimal control problem, where the objective is to maximize returns or minimize potential risks.
Beyond the widely known mean-variance approach introduced by \cite{markowitz1952}, which emphasizes risk-return trade-offs, the expected log-optimal growth~(ELG) strategy, originally introduced by~\cite{Kelly_1956}, is another influential approach. The ELG strategy can be viewed as a feedback policy that seeks to maximize the expected logarithmic growth of the investor's wealth.  
This approach aligns to optimize the long-term performance of a financial stochastic system; e.g., see \cite{primbs2007portfolio}  and \cite{barmish2024jump} for a tutorial on stock trading from a control-theoretic viewpoint.
Several foundational studies, such as~\cite{Breiman_1961, Cover_1984} and~\cite{cover2012elements}, have established that the ELG strategy is asymptotically and comparatively optimal compared to other objective functions. See also an extensive textbook \cite{maclean2011kelly} summarized the advantages and disadvantages of the ELG strategy.

Many extensions of the ELG problem have been explored. For example,~\cite{kuhn2010analysis} analyzed continuous-time log-optimal portfolios in both high-frequency trading and buy-and-hold scenarios. Recently,~\cite{hsieh2023asymptotic, wong2023frequency} incorporated rebalancing frequency and transaction costs in the log-optimal portfolio framework. However, in practice, the true distribution of returns is often unknown or only partially observed, which contrasts with the typical assumptions made in these studies.

\subsection{Ambiguity Consideration}

Emerging research trends have incorporated the concept of \emph{ambiguity} into return models to better capture real-world uncertainties.  Ambiguity refers to uncertainty about the true probability distribution of returns, even when certain properties are known or estimated from historical data. It quantifies the allowable deviations of these parameters from their nominal values.

To address this ambiguity issue, a significant area of robust optimization has focused on the worst-case approach. Comprehensive surveys of robust optimization and its applications can be found in works by \cite{ben2008selected, bertsimas2011theory}; see also \cite{rujeerapaiboon2016robust} and \cite{proskurnikov2023benefit} on a study of robust optimal growth problems. However, worst-case solutions are often overly conservative in practice. 
This challenge has spurred extensive research into distributionally robust optimization (DRO), which seeks a balance between robustness and optimality rather than focusing exclusively on worst-case scenarios.

\subsection{Distributionally Robust Optimization Approach}

 In decision-making under uncertainty in $\mathbb{R}^m$ space, there are three main frameworks: Stochastic Optimization~(SO), Robust Optimization (RO), and Distributionally Robust Optimization (DRO); see~\cite{sheriff2023nonlinear, shapiro2021lectures}.

Let $U$ be the utility function, $w$ be the decision variable, and $\mathcal{W}$ be the feasible set of $w$. SO aims to maximize~$\mathbb{E}^{\mathbb{F}_0}[U(w, X)] $, assuming investors have complete knowledge of the nominal distribution $ \mathbb{F}_0$ of the random variable $ X $. However, as mentioned previously, in reality, the distribution $\mathbb{F}_0$ is typically unknown and may change over time. These limitations necessitate solving stochastic optimization problems under worst-case scenarios, transforming the original stochastic problems into deterministic max-min (or min-max) problems, specifically through RO. Indeed, RO seeks to solve:
$
\max_{w \in \mathcal{W}} \; \inf_{x \in \mathfrak{X}} U(w, x),
$ 
where $ \mathfrak{X} $ is the support set of~$ X $. Although RO provides a robust solution, it can be overly conservative and computationally inefficient in practice; see~\cite{ben2009robust}.

To address these drawbacks, DRO serves as a bridge between SO and RO by incorporating an ambiguity set $\mathcal{P}$ for the distribution~$\mathbb{F}$. The goal is to solve
$
\max_{w \in \mathcal{W}}   \inf_{\mathbb{F} \in \mathcal{P}} \mathbb{E}^{\mathbb{F}}[ U(w, X)].
$
It is worth noting that the robustness of the solution depends on the size of the ambiguity set~$\mathcal{P}$. A larger~$\mathcal{P}$ reduces DRO to the aforementioned RO problem, while a smaller~$\mathcal{P}$ that contains only the nominal distribution~$\mathbb{F}_0$ simplifies DRO to SO.

As previously mentioned, the DRO formulation can be viewed as a generalization of both SO and RO. Various approaches exist for defining the ambiguity set in a DRO model. One common method is to represent the ambiguity set as a ball in the space of probability distributions, where the conservativeness of the optimization problem is controlled by adjusting the radius of the ambiguity set. Examples include the Prohorov metric~\cite{erdougan2006ambiguous}, the Kullback-Leibler divergence~\cite{hu2013kullback}, and the popular Wasserstein metric~\cite{mohajerin2018data, gao2023distributionally, li2023wasserstein}. Other methods, such as the total variation distance~\cite{rahimian2022effective} and \cite{farokhi2023distributionally}, consider the average absolute difference of distributions at each point, while the moment-based distance~\cite{delage2010distributionally} focuses on higher-order statistical properties. 
These methods provide different perspectives for measuring differences between distributions. A comprehensive review can be found in \cite{rahimian2019distributionally}.

Recently, several studies have focused on \emph{data-driven} DRO problems in finance applications, characterizing the ambiguity set based on historical data.
For instance,~\cite{sun2018distributional} explored a data-driven distributional robust Kelly strategy in a discrete distribution with a finite-support setting. Similarly,~\cite{hsieh2023solving} introduced a supporting hyperplane approximation approach, transforming a class of distributionally robust log-optimal portfolio problems with polyhedron ambiguity sets into tractable robust linear programs.
\cite{blanchet2022distributionally}  studied the distributionally robust mean-variance portfolio selection with Wasserstein distances.  Moreover,~\cite{li2023wasserstein} proposed optimizing the Wasserstein-Kelly portfolio, focusing instead on the log-return distribution. 
However, this approach relies on additional logarithmic transformation and does not consider transaction costs in the model.

While significant advances in DRO methodologies have been achieved, critical gaps remain: the robustly survivable trades across ambiguous distributions and the impact of transaction costs---key factors in real-world trading---are often overlooked. To bridge this gap between theory and practice, we propose a novel, cost-sensitive, data-driven, distributionally robust log-optimal portfolio optimization problem formulation. Notably, the Wasserstein metric, which allows for both discrete and continuous distributions in data-driven settings, provides three key advantages: finite sample guarantee, asymptotic consistency, and tractability, as discussed in~\cite{mohajerin2018data}. Leveraging these properties, we define the ambiguity set using the Wasserstein metric and incorporate a general convex transaction cost model into our proposed formulation.

\subsection{Transaction Costs Consideration}
In control systems, costs associated with control actions, such as energy consumption or actuator wear, are critical factors that influence the design of optimal control strategies. Similarly, in financial markets, transaction costs---such as commission fees, bid-ask spreads, latency costs, and transaction taxes,---play a crucial role in the optimal decision-making process.  

Previous studies have attempted to address transaction costs within control-theoretic frameworks. For example, \cite{bertsimas1998optimal} studied an optimal control strategy to minimize the mean-variance trade-off of the execution cost.
Other studies, such as \cite{oksendal2002optimal}, examined portfolio management with two assets, focusing on the effects of fixed and proportional transaction costs on fund transfers. \cite{lobo2007portfolio} addressed issues involving fixed and linear transaction costs.  \cite{mei2018portfolio} proposed suboptimal rebalancing and consumption policies under proportional transaction costs for mean-variance criterion, along with an upper bound analysis on policy performance. 
An empirical study by~\cite{ruf2020impact} analyzed portfolio performance in the presence of transaction costs across various classical portfolios.

Given the importance of transaction costs, this paper extends existing frameworks by focusing on a general convex transaction cost model. This model, commonly encountered in real-world applications, provides a more flexible and realistic representation of cost structures, such as linear or piecewise-linear transaction~costs, see~\cite{davis1990portfolio, korn1998portfolio, lobo2007portfolio}.  By incorporating these into our DRO framework, we obtain more robust solutions that reflect practical trading considerations.

\subsection{Contributions of This Paper}

Building on the control-theoretic perspective, we formulate a novel cost-sensitive, distributionally robust log-optimal portfolio optimization problem using the Wasserstein metric along with a general convex transaction cost model;  see  Section~\ref{section: Problem Formulation}.
We derive conditions that ensure robust survivability of trades for all distributions in the Wasserstein ball under convex transaction costs; see Lemma~\ref{lemma: robust survival condition with transaction costs rebalancing}.
Using duality theory, we further show that the infinite-dimensional primal distributionally robust optimization problem can be approximated by a finite convex program; see Theorem~\ref{theorem:convex reduction with non-zero transaction costs during rebalancing} and Corollary~\ref{corollary: Reduction of Semi-Infinite Constraint} in Section~\ref{section: Main Results}.  This enables tractable computation of the optimal control policy.
Additionally, we conduct extensive empirical studies in Section~\ref{section: Empirical Studies}, exploring the impact of adjusting the radius size of the ambiguity set and the effect of varying transaction costs.

\section{Preliminaries} \label{section: Problem Formulation}
This section provides some necessary preliminaries and the cost-sensitive distributionally robust log-optimal portfolio problem formulation.

\subsection{Notations}
Throughout the paper, we should denote by $\mathbb{R}$  the set of real numbers, $\mathbb{R}_{+}$ the set of nonnegative real numbers, and $\overline{\mathbb{R}} := \mathbb{R} \cup \{-\infty, \infty\}$ the extended reals.
All the random objects are defined in a probability space~$(\Omega, \mathcal{F}, \mathbb{P}).$ 
The notation $\| \cdot \|$ refers to the norm on~$\mathbb{R}^m$, then $\| \cdot \|_*$ is the associated dual norm defined through~$\| z \|_* := \max_x \{  z^\top x  : \| x \|\leq 1 \}$ for $z \in \mathbb{R}^m$.
We denote by $\delta_x$ the Dirac distribution concentrating unit mass at~$x \in \mathbb{R}^m.$ The product of two probability distributions~$\mathbb{P}_1$ and $\mathbb{P}_2$ on~$\mathfrak{X}_1$ and $\mathfrak{X}_2$, respectively, is the distribution $\mathbb{P}_1 \otimes \mathbb{P}_2$ on $\mathfrak{X}_1 \times \mathfrak{X}_2.$

\subsection{Cost-Sensitive Account Dynamics}	
Consider a financial market with $\mathfrak{M}  > 1$ assets and form a portfolio of $m \leq  \mathfrak{M}$ assets. For stage $k = 0,1,\dots,$ let $S_i(k) > 0$ with $i \in \{1, 2, \dots, m\}$ be the stock prices for the $i$th asset at time $k$. We denote the prices of these assets at time~$ k$ by a random vector ${S}(k) := [S_{1}(k) \ S_{2}(k) \ \cdots \ S_{m}(k)]^\top$.

For $k=0$, let the initial account value $V(0) > 0$ be given. Then the initial investment policy for the $i$th asset is of the linear feedback form $u_i(0) := w_i(0) V(0)$. The overall investment at $k=0$ is $u(0) = \sum_{i=1}^m u_i(0)$.

Let $n \geq 1$ denote the number of steps between rebalancing periods. The investor waits for $n$ steps and updates the investments to~$u_i(n) := w_i(n) V(n)$. In other words, the investor buys and holds the assets for $n$ steps before rebalancing.

The cost-sensitive account value dynamics at time $n \geq 1$ is characterized by the following stochastic difference equation:
\begin{align} \label{eq: freq-dependent account dynamics}
V(n) 
&= V(0) + \sum_{i=1}^{m} \frac{u_i(0)}{S_i(0)} (S_i(n) - S_i(0)) - TC(u(0))
\end{align}
where $u_i(0) := w_i(0) V(0)$ with $w_i(0)$ represents the portfolio weight\footnote{Note that the quantity $u_i(0)/S_i(0)$ corresponds to the number of shares traded for Asset $i$ at stage $k=0$.} and 
$
TC(u(0)) := \sum_{i=1}^m TC_i(u_i(0)) \in [0,V(0))
$
with~$TC_i(\cdot)$ being a continuous convex function representing the transaction costs for asset $i$ based on the trade amount~$u_i$. This cost model reflects realistic cost structures that cover a wide range of real-world transaction cost behaviors such as linear or piecewise-linear transaction~costs; see~\cite{davis1990portfolio, korn1998portfolio}. 

\begin{remark} \rm
The parameter $n \geq 1$ captures the impact of rebalancing frequency where $n=1$ corresponds to the highest possible rebalancing frequency, but it may accumulate more transaction costs.
As $n$ increases, rebalancing occurs less frequently, which reduces transaction costs.
\end{remark}

Set $w_i := w_i(0)$ and  we take portfolio weight vector~by 
$$
w := [w_{1}\  \cdots\ w_{n}]^\top \in \mathcal{W} := \left\{ w \in \mathbb{R}^m : \mathbf{1}^\top w = 1, w \succeq 0 \right\}
$$
where $\mathbf{1}$ is the one-vector and $w \succeq 0$ means $w_i \geq 0$ for all~$i = 1, \dots, m$.

To reflect the effect of portfolio rebalancing, we 
follow previous research in~\cite{hsieh2023asymptotic} and work with the $n$-period compound returns~$\mathcal{X}_{n,i}$ of each asset $i \in \{1, 2, \dots, m\}$, defined as 
$$
\mathcal{X}_{n,i} : = \frac{S_i(n) - S_i(0)}{S_i(0)} = \prod_{k=0}^{n-1}(1+X_i(k))-1,
$$
where $X_{i}(k):=\frac{S_i(k+1) - S_i(k)}{S_i(k)}$. We denote the random vector of $n$-period compound returns of $m$ assets, by~$\mathcal{X}_n:=[\mathcal{X}_{n,1} \ \mathcal{X}_{n,2} \ \cdots\ \mathcal{X}_{n,m}]^\top$ and define the support set of $\mathcal{X}_n$ by
\begin{align} \label{eq: support set of compound return}
\mathfrak{X} :=\{\mathcal{X}_n: \mathcal{X}_{\min, i} \leq \mathcal{X}_{n, i} \leq \mathcal{X}_{\max,i} ,\ i = 1, \dots, m\},
\end{align}
where $\mathcal{X}_{\min,i} := ( 1 + X_{\min,i})^{n} - 1 > -1$ and $\mathcal{X}_{\max, i} :=(1+X_{\max, i})^{n} - 1 < \infty$ for all $n\geq 1$.
In the sequel, we treat the number of steps between rebalancing, $n \geq 1$, as a predetermined parameter that is set by the investor.

Substituting control or investor's policy $u_i(0) = w_i(0) V(0)$ into the above stochastic equation~\eqref{eq: freq-dependent account dynamics}, we obtain
{\small
\begin{align} 
	V(n) 
	%&= V(0) + \sum_{i=1}^{m} w_i V(0) \cdot \frac{S_i(n) - S_i(0)}{S_i(0)} - \sum_{i=1}^{m} w_i V(0) \cdot c_i \notag \\
	&= V(0) \left( 1 + \sum_{i=1}^{m} w_i(0) \left(\frac{S_i(n) - S_i(0)}{S_i(0)} \right) - \frac{TC(u(0))}{V(0)} \right) \notag \\
	&= V(0) \left( 1 + w^\top  {\mathcal{X}}_n  - \frac{{TC}(u(0))}{V(0)}\right) \notag \\
	&= V(0) \left( c(w)  + w^\top  {\mathcal{X}}_n \right)  \label{eq: log-optimal account value with transaction cost rebalancing}
\end{align}
}where $c(w) : = 1 - \frac{{TC}(u(0))}{V(0)}$, which is a deterministic concave function. To emphasize the dependence on portfolio weight $w$, we may sometimes write $V_w(n)$ instead of~$V(n)$. 
In practice, before deriving any optimal policy, it is important to ensure that the trade can never go bankrupt. This issue is addressed in Section~\ref{subsection: Robustly Survivable Trades and Feasible Controls}.

\subsection{Wasserstein Ambiguity Set}
As mentioned in Section~\ref{section: Introduction}, the true distribution of asset returns is often unknown or estimated from historical data. Traditional robust optimization tends to be overly conservative by focusing on worst-case scenarios. In contrast, the Wasserstein metric offers a more flexible, data-driven approach to account for distributional ambiguity.
Specifically, let~$\mathcal{M}( { \mathfrak{X} } )$  be the space of all probability distributions~$\mathbb{F}$ supported on~$  { \mathfrak{X} } $ with $\mathbb{E}^\mathbb{F}[ \| {\mathcal{X}}_n\| ] = \int \| {\mathcal{X}}_n\| d\mathbb{F} < \infty$. We are ready to define the Wasserstein metric as follows:

\begin{definition}[Wasserstein Metric] \label{definition:Wasserstein metric with transaction costs during rebalancing} \rm
	Given any two distributions $\mathbb{F}^1, \mathbb{F}^2 \in \mathcal{M}( { \mathfrak{X} })$, the (Type-1) \emph{Wasserstein metric}~$d: \mathcal{M}( { \mathfrak{X} }) \times \mathcal{M}( { \mathfrak{X} }) \to \mathbb{R}$ is defined~by
	\begin{align*}
		d(\mathbb{F}^1, \mathbb{F}^2)
		& := \inf_\Pi \left\{ \left( \mathbb{E}_{( \mathcal{X}^1, \mathcal{X}^2) \sim \Pi} [\|\mathcal{X}^1 - \mathcal{X}^2\|] \right) \right\}
	\end{align*}
	where $\Pi$ is a joint distribution of $\mathcal{X}^1$ and $\mathcal{X}^2$ with marginal distributions $\mathbb{F}^1$ and $\mathbb{F}^2$, respectively.
\end{definition}

Consistent with~\cite{calafiore2013direct}, \cite{mohajerin2018data} and \cite{li2023wasserstein},
suppose that $\widehat{\mathcal{X}}_1, \dots, \widehat{\mathcal{X}}_N$ represent $N$ random samples of finite random vector~${\mathcal{X}}_n \sim \mathbb{F}$. The associated empirical distribution, denoted by $\widehat{\mathbb{F}}$ is given by  $\widehat{\mathbb{F}} := \frac{1}{N} \sum_{j=1}^{N} \delta_{\widehat{\mathcal{X}}_j}$ where $\delta_{\widehat{\mathcal{X}}_j}$ is the Dirac delta function. Henceforth, for notational simplicity, we may use $\widehat{\mathcal{X}}$ to denote a typical element of the iid random samples instead of specifying~$\widehat{\mathcal{X}}_j$ for~$j \in \{1,\dots, N\}$. Similarly, we may sometimes denote~${\mathcal{X}}_n$ as~${\mathcal{X}}$.

Let $\varepsilon \geq 0$. Using the  Wasserstein metric in Definition~\ref{definition:Wasserstein metric with transaction costs during rebalancing}, we construct the ambiguity set  consisting of distributions of radius $\varepsilon$ centered at the empirical distribution~$\widehat{\mathbb{F}}$, i.e., the \emph{Wasserstein ball}, call it  $\mathcal{B}_\varepsilon( \widehat{\mathbb{F}} )$, defined by
\begin{align} \label{def: Wasserstein ball}
\mathcal{B}_\varepsilon( \widehat{\mathbb{F}} ) := \left\{ \mathbb{F} \in \mathcal{M}( { \mathfrak{X} } ) : d(\mathbb{F}, \widehat{\mathbb{F}}) \leq \varepsilon \right\}.
\end{align}

\begin{remark} \rm
Note that if $\varepsilon = 0$, then $\mathcal{B}_\varepsilon( \widehat{\mathbb{F}} )  = \{ \widehat{\mathbb{F}}\}$, a singleton empirical distribution.
The reader is referred to \cite{mohajerin2018data, xie2024distributionally} and the references therein for further details on this topic.
\end{remark}

\subsection{Robustly Survivable Trades and Feasible Controls} \label{subsection: Robustly Survivable Trades and Feasible Controls}
Similar to the stability requirement in classical control systems, 
ensuring that trades are survivable, i.e.,~$ \mathbb{P}^\mathbb{F} (V(n) > 0) = 1 $ for all $n \geq 1$ is crucial in financial systems.
We extend the concept of a ``survivable trade'' to a ``robustly survivable trade'' by leveraging the Wasserstein metric. 

%\begin{definition}[Robustly Survivable Trades]
%Let $ V(0) > 0 $ and $ w \in \mathcal{W} $. A trade is called \emph{robustly survivable} if,  for all distributions $ \mathbb{F} $ in the Wasserstein ball~$ \mathcal{B}_\varepsilon( \widehat{\mathbb{F}} ) $, the positive account value is maintained positive with probability one, i.e.,
%\[
%\mathbb{P}^\mathbb{F}( V(n) > 0 ) = 1 \quad \text{for all } n \geq 1.
%\]
%\end{definition}

\begin{lemma}[Robustly Survivable Trades] \label{lemma: robust survival condition with transaction costs rebalancing} \it 
Let~$ V(0) > 0 $ and $w \in \mathcal{W}$ be given. If the weight vector $ w $ satisfies
$
\sum_{i=1}^{m} w_i {\mathcal{X}}_{\min, i} + c(w) > 0,
$
then, for any $\varepsilon \geq 0$ and for all distributions $ \mathbb{F} \in  \mathcal{B}_\varepsilon( \widehat{\mathbb{F}} ) $, we have
\[
\mathbb{P}^\mathbb{F}( V(n) > 0 ) = 1 \; \text{ for all } n \geq 1.
\]
\end{lemma}

\begin{proof}
See Appendix~\ref{appendix: proofs in section formulation}.
\end{proof}

\begin{remark}
In the context of the cost-sensitive DRO problem, ensuring robust survivability means that the investor can select portfolio weights $ w $ that guarantee the account value remains positive regardless of the true underlying distribution within the ambiguity set. This adds an extra layer of protection against model uncertainty and adverse market conditions.
\end{remark}

Following Lemma~\ref{lemma: robust survival condition with transaction costs rebalancing}, we further require the portfolio weight vector $w:=[w_1\ \cdots\ w_m]^\top \in \mathcal{W}_{\rm s}$ with $\mathcal{W}_{\rm s}$ defined~as 
$$
\mathcal{W}_{\rm s} := \mathcal{W} \cap \left\{ w \in \mathbb{R}^m :  \sum_{i=1}^{m} w_i {\mathcal{X}}_{\min, i} + c(w) > 0 \right\}.
$$
It is obvious that $\mathcal{W}_{\rm s} \subseteq \mathcal{W}$.
In the following research, we consider the closure of $\mathcal{W}_{\rm s}$, denoted as $\overline{\mathcal{W}}_{\rm s}$.

\subsection{Cost-Sensitive DRO Problem}

\begin{problem}[Cost-Sensitive DRO Problem] \rm
Fix $n \geq 1$ and $\varepsilon \geq 0$.
The cost-sensitive distributionally robust log-optimal portfolio problem with the Wasserstein metric is given by
\begin{align*}
	&\max_{w \in \overline{\mathcal{W}}_{\rm s}} \; 
	\inf_{\mathbb{F} \in \mathcal{B}_\varepsilon( \widehat{\mathbb{F}} )}\, \frac{1}{n}\, \mathbb{E}^{\mathbb{F}} \left[ \log \frac{V_w (n)}{V(0)} \right] 
	%	&{\rm s.t. } \ V(n) = V(0) (c(w) + w^\top \mathcal{X}_n)
\end{align*}
With Equation~\eqref{eq: log-optimal account value with transaction cost rebalancing}, the above problem can be reduced 
\begin{align}\label{problem: DRO ELG with transaction costs during rebalancing}
	&\max_{w \in \overline{\mathcal{W}}_{\rm s}}\,{\inf_{\mathbb{F} \in \mathcal{B}_\varepsilon( \widehat{\mathbb{F}} )}}\, \frac{1}{n}\, \mathbb{E}^{\mathbb{F}} \left[\log \left(c(w) + w^\top {\mathcal{X}}_n \right) \right].
\end{align}
\end{problem}
Note that the inner minimization problem seeks to find an optimal distribution $\mathbb{F}$ over the Wasserstein ball, which is, in principle, an infinite-dimensional optimization problem.

\begin{remark} \rm
Suppose that the transaction cost model takes a proportional cost structure with $TC_i(u_i(0)):= c_i w_i V(0)$.
When the radius $\varepsilon = 0$, indicating no ambiguity exists and full trust is placed in historical data, the problem reduces to solving the classical log-optimal problem $\max_{w \in \mathcal{W}}\, \frac{1}{n}\, \mathbb{E}^{\widehat{\mathbb{F}}} [ \log ( 1  + w^\top {\mathcal{X}} ) ]$, see \cite{hsieh2023asymptotic, wong2023frequency}.
\end{remark}

\section{Theoretical Results} \label{section: Main Results}
This section shows that the infinite-dimensional cost-sensitive distributional robust log-optimal problem~\eqref{problem: DRO ELG with transaction costs during rebalancing} can be approximated to a finite convex program with a concave objective, thereby facilitating computational tractability.

\begin{theorem}[Convex Approximation] \label{theorem:convex reduction with non-zero transaction costs during rebalancing} \it
Fix $n \geq 1.$ For any $\varepsilon > 0$, the cost-sensitive DRO log-optimal problem~\eqref{problem: DRO ELG with transaction costs during rebalancing} can be approximated to the convex program with concave objective function:
\begin{align} 
	&\sup_{w, \lambda, s_j, z_j}\, \frac{1}{n}\, \left(- \lambda\varepsilon + \frac{1}{N} \sum_{j=1}^{N} {s}_j \right) \notag \\ 
	&\text{\rm s.t. }  \
	\medmath{ \min_{ {\mathcal{X}} \in { \mathfrak{X} }} \left[ \log \left( c(w) + w^\top {\mathcal{X}}\right) +  z_j^\top ( {\mathcal{X}} -\widehat{\mathcal{X}}_j ) \right] \geq  {s}_j, \;  \forall j } \label{problem:DRO-ELG dual with transaction costs during rebalancing} \\
	& \qquad \lambda \geq 0, \notag\\ 
	& \qquad { \| z_j\|_* \leq \lambda}, \; j=1, \dots, N \notag,\\
	& \qquad w \in \overline{\mathcal{W}}_{\rm s} \notag
\end{align}
where $z_j \in \mathbb{R}^m$ and $\| z_j\|_*$ is the dual norm of~$z_j$.
\end{theorem}

\begin{proof}
Using the Wasserstein distance described in Definition~\ref{definition:Wasserstein metric with transaction costs during rebalancing}, we first re-express the inner minimization problem in Equation~\eqref{problem: DRO ELG with transaction costs during rebalancing} using the duality technique. Specifically, we note that
\begin{align}
	&	\inf_{\mathbb{F} \in \mathcal{B}_\varepsilon( \widehat{\mathbb{F}})} \mathbb{E}^{\mathbb{F}} \left[\log ( c(w) + w^\top {\mathcal{X}} ) \right] \label{problem: primal infinite-convex (linear program) problem}
	\\
	&\qquad = \begin{cases}
		\displaystyle\inf_{ \mathbb{F} \in \mathcal{M}( { \mathfrak{X} }) }\, \displaystyle\int \log \left( c(w) + w^\top  {\mathcal{X}}\right) d\mathbb{F}\\
		\text{s.t. } 
		d( \mathbb{F}, \widehat{\mathbb{F}}) \leq \varepsilon\\
	\end{cases} \notag \\
	&\qquad = \begin{cases}
		\displaystyle\inf_{\Pi, \mathbb{F}}\, \displaystyle\int \log \left( c(w) + w^\top  {\mathcal{X}}\right) d\mathbb{F}\\
		\text{s.t. } 
		\int \| {\mathcal{X}} - \mathcal{X}^{\prime} \| d\Pi \leq \varepsilon\\
	\end{cases}\label{problem:the inner minimization problem with transaction costs during rebalancing}
\end{align}
where $\Pi$ is a joint distribution of ${\mathcal{X}}$ and $\mathcal{X}^{\prime}$ with marginals $\mathbb{F}$ and $\widehat{\mathbb{F}}$, respectively. 
Following the fact that $\widehat{\mathbb{F}} = 
\frac{1}{N} \sum_{j=1}^N \delta_{\widehat{\mathcal{X}}_j}$ and the law of total probability, the joint probability distribution $\Pi$ of ${\mathcal{X}}$ and $\mathcal{X}^{\prime}$ can be constructed from the marginal distribution $\widehat{\mathbb{F}}$ of $\mathcal{X}^{\prime}$ and the conditional distributions $\mathbb{F}_j$ of ${\mathcal{X}}$ given $\mathcal{X}^{\prime} = \widehat{\mathcal{X}}_j$ for~$j \leq N$. That is, we have 
$
\Pi = \frac{1}{N}\sum_{j=1}^{N}\delta_{\widehat{\mathcal{X}}_j} \otimes \mathbb{F}_j
$
and 
$\mathbb{F} = \frac{1}{N}\sum_{j=1}^N {\mathbb{F}}_j$; see also Lemma~\ref{lemma: Decomposition of Marginal Distribution} for a detailed justification. 
Hence, Problem~\eqref{problem:the inner minimization problem with transaction costs during rebalancing} can be further rewritten as
\begingroup
\allowdisplaybreaks
\begin{align}  
	&\begin{cases}
		\displaystyle \inf_{\mathbb{F}_j \in \mathcal{M}( \mathcal{ { \mathfrak{X} } } ), \, \forall j} \, \frac{1}{N}\displaystyle\sum_{j=1}^{N} \int \log ( c(w) + w^\top  {\mathcal{X}} ) d\mathbb{F}_j \\
		\text{s.t. } \frac{1}{N}\displaystyle\sum_{j=1}^{N} \int \| {\mathcal{X}} - \widehat{\mathcal{X}}_j\|  d\mathbb{F}_j \leq \varepsilon
	\end{cases}\\
	%\end{align}
	%Introducing a Lagrange multiplier $\lambda$, we can obtain equivalent the standard duality argument leads to   
	%\begin{align}
	&= \inf_{\mathbb{F}_j \in \mathcal{M}( { \mathfrak{X} }), \forall j} \, \sup_{\lambda \geq 0} \ \frac{1}{N} \sum_{j=1}^{N} \int \log ( c(w) + w^\top  {\mathcal{X}} )d{\mathbb{F}_j}    \notag \\
	& \qquad \qquad \qquad  - \lambda \left( \varepsilon - \frac{1}{N} \sum_{j=1}^{N} \int \| {\mathcal{X}} - \widehat{\mathcal{X}}_j\| d{\mathbb{F}_j}\right)  \notag \\ 
	& \medmath{ \medmath{ = \inf_{\mathbb{F}_j \in \mathcal{M}( { \mathfrak{X} }), \forall j}\, \sup_{\lambda \geq 0} \left\{ -\lambda \varepsilon + \frac{1}{N} \sum_{j=1}^{N} \int \left[\log (c(w) + w^\top {\mathcal{X}} ) + \lambda \| {\mathcal{X}} - \widehat{\mathcal{X}}_j\| \right] d \mathbb{F}_j \right\} }  } \notag \\ 
	&  \medmath{  \medmath{ = \sup_{\lambda \geq 0}\, \inf_{\mathbb{F}_j \in \mathcal{M}( { \mathfrak{X} }), \forall j} \left\{-\lambda \varepsilon + \frac{1}{N} \sum_{j=1}^{N} \int \left[\log ( c(w) + w^\top {\mathcal{X}} ) + \lambda \| {\mathcal{X}} - \widehat{\mathcal{X}}_j\| \right] d{\mathbb{F}_j} \right\} } } \label{eq:turn into inf_sup with transaction costs during rebalancing}\\ 
	&  \medmath{  \medmath{ = \sup_{\lambda \geq 0}\,  \left\{-\lambda \varepsilon + \frac{1}{N} \sum_{j=1}^{N} \inf_{\mathbb{F}_j \in \mathcal{M}( { \mathfrak{X} })} \int \left[\log ( c(w) + w^\top {\mathcal{X}} ) + \lambda \| {\mathcal{X}} - \widehat{\mathcal{X}}_j\| \right] d{\mathbb{F}_j} \right\} } }  \notag \\ 
	&  \medmath{ = \sup_{\lambda \geq 0} \left\{ -\lambda\varepsilon + \frac{1}{N} \sum_{j=1}^{N} \inf_{ {\mathcal{X}} \in { { \mathfrak{X} }}} \left[\log \left( c(w) + w^\top  {\mathcal{X}}\right) + \lambda {\| {\mathcal{X}} - \widehat{\mathcal{X}}_j\|} \right] \right\}}, \label{eq:before auxilary variables with transaction costs during rebalancing} 
\end{align}
\endgroup
%where Inequality~\eqref{eq:turn into inf_sup with transaction costs during rebalancing} follows from the minimax inequality.
where~\eqref{eq:turn into inf_sup with transaction costs during rebalancing} holds since the primal problem~\eqref{problem: primal infinite-convex (linear program) problem}, as an infinite-dimensional convex program, has a strictly feasible solution $\widehat{ \mathbb{F} }$ that satisfies 
$d(\widehat{ \mathbb{F} }, \widehat{\mathbb{F}} ) = 0 < \varepsilon$. Hence, Slater's condition holds and strong duality follows.
Moreover, Equality~\eqref{eq:before auxilary variables with transaction costs during rebalancing} is proven by using the fact that~$\mathcal{M}( {  \mathfrak{X} })$ contains all the Dirac distribution on $ {  \mathfrak{X} }$; see Lemma~\ref{lemma: equity before auxilary variables with transaction costs during rebalancing}.

Now, by introducing epigraphical auxiliary variables 
$ {s}_j$ for~$j = 1, \dots, N$, we rewrite Equation~\eqref{eq:before auxilary variables with transaction costs during rebalancing} as follows:
\begin{align}
	&\begin{cases}
		\displaystyle\sup_{\lambda,  {s}_j} \, -\lambda\varepsilon + \frac{1}{N} \displaystyle\sum_{j=1}^{N}  {s}_j \\
		\text{s.t. } \displaystyle\inf_{ {\mathcal{X}} \in  { \mathfrak{X} }}  \left[ \log \left(c(w) + w^\top  {\mathcal{X}}\right) + \lambda {\| {\mathcal{X}} - \widehat{\mathcal{X}}_j\|} \right] \geq  {s}_j, \, \forall j \\
		\qquad \lambda \geq 0 \notag
	\end{cases} \\
	& = \begin{cases} 
		\label{eq: double dual norm representation}
		\displaystyle\sup_{\lambda, {s}_j} \, -\lambda\varepsilon + \frac{1}{N} \displaystyle\sum_{j=1}^{N}  {s}_j \\
		\text{s.t. } \medmath{ \medmath{ \inf_{ {\mathcal{X}} \in  { \mathfrak{X} }}  \left[\log \left( c(w) + w^\top  {\mathcal{X}} \right) + \max_{\| z_j \|_* \leq \lambda} z_j^\top \left( {\mathcal{X}} - \widehat{\mathcal{X}}_j \right) \right] \geq  {s}_j , \, \forall j} }  \\
		\qquad \lambda \geq 0  
	\end{cases}
\end{align}
where Equality~\eqref{eq: double dual norm representation} holds by using the fact that the double dual norm is equal to the ordinary norm in finite-dimensional space, i.e., $\|a\|_{**} = \| a \|$ for $a \in \mathbb{R}^m$, see \cite{beck2017first}.
Hence, we rewrite~\eqref{eq: double dual norm representation} further as follows: 
\begingroup
\allowdisplaybreaks
\begin{align}
	%	&\begin{cases}
		%		\displaystyle\sup_{\lambda, {s}_j} \, -\lambda\varepsilon + \frac{1}{N} \displaystyle\sum_{j=1}^{N}  {s}_j \\
		%		\text{s.t. } \medmath{ \medmath{ \inf_{ {\mathcal{X}} \in  { \mathfrak{X} }}  \left[\log (c(w) + w^\top  {\mathcal{X}} ) + \sup_{\| z_j \|_* \leq \lambda} z_j^\top ( {\mathcal{X}} - \widehat{\mathcal{X}}_j) \right] \geq  {s}_j , \, \forall j} }  \\
		%		\qquad \lambda \geq 0  \notag
		%	\end{cases}\\
	&\begin{cases} % \label{eq: dual norm constraint}
		\displaystyle\sup_{\lambda, {s}_j}  \, -\lambda\varepsilon + \frac{1}{N} \displaystyle\sum_{j=1}^{N}  {s}_j \\
		\text{s.t. } \medmath{ \medmath{ \inf_{ {\mathcal{X}} \in  { \mathfrak{X} }}  
				\displaystyle \max_{ \|z_j\|_* \leq \lambda} \left[ \log \left( c(w) + w^\top  {\mathcal{X}} \right) +  z_j^\top ( {\mathcal{X}} - \widehat{\mathcal{X}}_j) \right] \geq  {s}_j , \, \forall j} } \\
		\qquad \lambda \geq 0   \notag
	\end{cases}\\
	&	\geq \begin{cases} \label{ineq:turn into min_sup with transaction costs during rebalancing}
		\displaystyle\sup_{\lambda, {s}_j} \, -\lambda\varepsilon + \frac{1}{N} \displaystyle\sum_{j=1}^{N}  {s}_j \\
		\text{s.t. } \medmath{ \medmath{ 
				\displaystyle \max_{ \|z_j\|_* \leq \lambda} \inf_{ {\mathcal{X}} \in  { \mathfrak{X} }}   \left[ \log \left( c(w) + w^\top  {\mathcal{X}} \right) +  z_j^\top ( {\mathcal{X}} - \widehat{\mathcal{X}}_j) \right] \geq  {s}_j , \, \forall j} } \\
		\qquad \lambda \geq 0
	\end{cases} \\
	& = \begin{cases}
		\displaystyle \sup_{\lambda, {s}_j,z_j} \, -\lambda\varepsilon + \frac{1}{N} \displaystyle \sum_{j=1}^{N}  {s}_j \\
		\text{s.t. }   \
		\medmath{ \inf_{ {\mathcal{X}} \in  { \mathfrak{X} }} \left[ \log \left( c(w) + w^\top  {\mathcal{X}} \right) + z_j^\top ( {\mathcal{X}} - \widehat{\mathcal{X}}_j) \right] \geq  {s}_j, \, \forall j} \\
		\qquad \lambda \geq 0 \\
		\qquad { \|z_j\|_* \leq \lambda} , \quad j=1,\dots,N \notag
	\end{cases}
\end{align}
\endgroup
where Inequality~\eqref{ineq:turn into min_sup with transaction costs during rebalancing} holds via minimax inequality on the constraint.
Since the quantity $\log \left(c(w) + w^\top  x \right) + z_j^\top ( x - \widehat{\mathcal{X}}_j)$ is continuous in $ x$ over a compact domain~$ { \mathfrak{X} }$, the Weierstrass Extreme Value Theorem indicates that the minimum is attained and the infimum operator in the constraint above can be replaced by the minimum operator. Therefore, in combination with the arguments above, we have
\begin{align}
	%	 \label{problem:primal dual problem of DRO ELG with transaction costs during rebalancing}
	\begin{cases}
		\displaystyle\sup_{\lambda, {s}_j,z_j} \, - \lambda\varepsilon + \frac{1}{N} \displaystyle\sum_{j=1}^{N}  {s}_j \\
		\text{s.t. }  \
		\medmath{\min_{ {\mathcal{X}} \in  { \mathfrak{X} }} \left[ \log \left( c(w) + w^\top  {\mathcal{X}}\right) +  z_j^\top ( {\mathcal{X}} - \widehat{\mathcal{X}}_j) \right] \geq  {s}_j , \; \forall j  } \\
		\qquad \lambda \geq 0 \\
		\qquad {\|z_j\|_* \leq \lambda} , \quad j=1,\dots,N \notag
	\end{cases} 
\end{align}
Hence, replacing the inner minimization problem~\eqref{problem: DRO ELG with transaction costs during rebalancing} by the above optimization problem, we obtain the approximation:
\begin{align} 
	&\sup_{w, \lambda, {s}_j, z_j}\, \frac{1}{n}\, \left( - \lambda\varepsilon + \frac{1}{N} \displaystyle\sum_{j=1}^{N}  {s}_j \right) \label{problem: approximate dual DRO} \\ 
	& \text{s.t. }  \
	\medmath{ \min_{ {\mathcal{X}} \in  { \mathfrak{X} }} \left[ \log \left( c(w) + w^\top  {\mathcal{X}} \right) +  z_j^\top ( {\mathcal{X}} - \widehat{\mathcal{X}}_j) \right] \geq  {s}_j , \quad \forall j} \notag\\
	&\qquad \lambda \geq 0 \notag\\ 
	&\qquad {\|z_j\|_* \leq \lambda} , \quad j=1,\dots,N \notag,\\
	& \qquad  w \in \overline{\mathcal{W}}_{\rm s} \notag
\end{align}
where   $z_j \in \mathbb{R}^m$ and $\| z_j\|_*$ is the dual norm of $z_j$.

To complete the proof, it remains to show that Problem~\eqref{problem: approximate dual DRO} is a convex program. 
Indeed, for fixed $\varepsilon > 0$, it is readily verified that $- \lambda\varepsilon + \frac{1}{N} \sum_{j=1}^{N}  {s}_j$ is affine in~$\lambda$ and~$ {s}_j$ for every $j \leq N$, which implies the objective function is a concave function. Additionally, using the robustly survivable trade condition from Lemma~\ref{lemma: robust survival condition with transaction costs rebalancing}, we have $c(w) + w^\top \mathcal{X} >0$. Moreover, since $c(w) + w^\top  {\mathcal{X}}$ is concave in~$w$ and $ {\mathcal{X}}$, taking the logarithm function, which is concave and increasing, yields a concave composition function; see \cite[Chapter~3.2.4]{boyd2004convex}. Next,~$z_j^\top ( {\mathcal{X}} - \widehat{\mathcal{X}}_j)$ is affine in $ {\mathcal{X}}$ and in $z_j$ for every $j \leq N$; hence is concave. Taking the minimum function preserves the concavity. Then, the minimum term forms a convex constraint set. Moreover, the constraints $\lambda \geq 0$  and $\|z_j\|_* \leq \lambda$ are convex. On the other hand, $\overline{\mathcal{W}}_{\rm s}$ is a convex set since the intersection of two convex sets preserves convexity. Hence, the above problem is a convex program, and the proof is~complete.
\end{proof}

The next corollary indicates that the semi-infinite convex approximation above can be further expressed in a finite convex formulation.

\begin{corollary}[Reduction of Semi-Infinite Constraint] \label{corollary: Reduction of Semi-Infinite Constraint}
Fix $n \geq 1.$ For any $\varepsilon > 0$, the convex approximation problem~\eqref{problem:DRO-ELG dual with transaction costs during rebalancing} is equivalent to the convex program with concave objective function:
\begin{align} 
	&\sup_{w, \lambda, s_j, z_j}\, \frac{1}{n}\, \left(- \lambda\varepsilon + \frac{1}{N} \sum_{j=1}^{N} {s}_j \right) \notag \\ 
	&\text{\rm s.t. }  \
	\medmath{ \min_{ {\mathcal{X}}^v \in {\rm Ext}({ \mathfrak{X} })} \left[ \log \left( c(w) + w^\top {\mathcal{X}}^v \right) +  z_j^\top ( {\mathcal{X}}^v - \widehat{\mathcal{X}}_j ) \right] \geq  {s}_j, \;  \forall j }   \\
	& \qquad \lambda \geq 0, \notag\\ 
	& \qquad { \| z_j\|_* \leq \lambda}, \; j=1, \dots, N \notag,\\
	& \qquad w \in \overline{\mathcal{W}}_{\rm s} \notag
\end{align}
where $z_j \in \mathbb{R}^m$, $\| z_j\|_*$ is the dual norm of~$z_j$, and ${\rm Ext}({ \mathfrak{X} })$ is the set of extreme points of $\mathfrak{X}$.
\end{corollary}

\begin{proof} 
Define the function:
\[
\phi_j (\mathcal{X}) := \log \left( c(w) + w^\top {\mathcal{X}}\right) +  z_j^\top ( {\mathcal{X}} -\widehat{\mathcal{X}}_j )
\]
which is a continuous concave function over a convex compact set $\mathfrak{X}$. Hence, the minimum is attained by the Weierstrass Extremum Theorem.
Notably, since $\mathfrak{X}$ is convex compact, according to the Krein-Milman Theorem, the set $\mathfrak{X}$ can be expressed as a convex hull of its extreme points. That is, 
$
\mathfrak{X} ={\rm conv}({\rm Ext}(\mathfrak{X}))
$ 
where~${\rm Ext}({ \mathfrak{X} })$ is the set of extreme points of $\mathfrak{X}$.
Hence, all points in~$\mathfrak{X}$ can be represented as convex combinations of its extreme points. 
Additionally, the concave function~$\phi_j(\mathcal{X})$ that attains a minimum over compact convex set $\mathcal{X}$ attains the minimum at some extreme point $\mathcal{X}^v$ of $\mathfrak{X}$, see \cite[Proposition 2.4.1]{bertsekas2009convex}.  Therefore, the semi-infinite constraint
\begin{align*}
	\medmath{ \min_{ {\mathcal{X}} \in { \mathfrak{X} }} \left[ \log \left( c(w) + w^\top {\mathcal{X}}\right) +  z_j^\top ( {\mathcal{X}} -\widehat{\mathcal{X}}_j ) \right] \geq  {s}_j, \;  \forall j }   
\end{align*}
can be reduced to 
\[
\medmath{ \min_{ {\mathcal{X}}^v \in {\rm Ext}({ \mathfrak{X} })} \left[ \log \left( c(w) + w^\top {\mathcal{X}}^v \right) +  z_j^\top ( {\mathcal{X}}^v - \widehat{\mathcal{X}}_j ) \right] \geq  {s}_j, \;  \forall j }   
\]
and the proof is complete.
\end{proof}

\begin{remark} \rm  
 While Corollary~\ref{corollary: Reduction of Semi-Infinite Constraint} shows that the infinite-dimensional DRO problem~\eqref{problem: DRO ELG with transaction costs during rebalancing} can be approximated by a finite convex program, the number of extremum points grows exponentially with the number of assets $m$, which may pose computational challenges. However, standard approaches such as the \emph{cutting plane method}, see \cite[Chapter 14.7]{luenberger2021linear}, can be used to efficiently solve the resulting convex program.
\end{remark}

\begin{lemma}[Decomposition of Distributions] \label{lemma: Decomposition of Marginal Distribution} \it
Let $\mathbb{F}$ and $ \widehat{\mathbb{F}}$ be the marginal distribution of $\mathcal{X}$ and ${ \mathcal{X}}'$, respectively, as defined in Section~\ref{section: Problem Formulation}. 
Then, the following two expression holds:
\begin{itemize}
  \item[$(i)$] We have $
    \mathbb{F} = \frac{1}{N} \sum_{j=1}^N \mathbb{F}_j
	$
  where $\mathbb{F}_j$ is the conditional distribution of ${\mathcal{X}}$ given $\mathcal{X}^{\prime} = \widehat{\mathcal{X}}_j$ for each $j =1,2,\dots, N$. 
	
  \item[$(ii)$] The joint distribution $\Pi$ of $\mathcal{X}$ and $\mathcal{X}'$ is given by $\Pi = \frac{1}{N}\sum_{j=1}^{N}\delta_{\widehat{\mathcal{X}}_j} \otimes \mathbb{F}_j$ where $\delta_{\widehat{\mathcal{X}}_j} \otimes \mathbb{F}_j$ is product measure.
\end{itemize}	
\end{lemma}

\begin{proof}
See Appendix~\ref{appendix: proofs in section main results}.
\end{proof}

\begin{lemma} \label{lemma: equity before auxilary variables with transaction costs during rebalancing} \it
For $\mathcal{M}( { \mathfrak{X} })$ contains all probability distribution supported on $ { \mathfrak{X} }$, the following equality holds:
\begin{align*}
	&\inf_{\mathbb{F}_j \in \mathcal{M}( { \mathfrak{X} })}  \int \left[\log \left(c(w) +w^\top  {\mathcal{X}}\right) + \lambda \| {\mathcal{X}}-\widehat{\mathcal{X}}_j\| \right] d{\mathbb{F}_j}
	\\
	&\qquad =
	\inf_{ {\mathcal{X}} \in  { \mathfrak{X} }} \left[ \log \left(c(w) + w^\top  {\mathcal{X}}\right) + \lambda {\| {\mathcal{X}} + \widehat{\mathcal{X}}_j\|} \right].
\end{align*}
\end{lemma}

\begin{proof}
See Appendix~\ref{appendix: proofs in section main results}.
\end{proof}

\begin{lemma} \label{lemma:turn into inf_sup with transaction costs during rebalancing}  
\it The dual problem~\eqref{problem:DRO-ELG dual with transaction costs during rebalancing} has a finite optimal value, and the set of dual optimal solutions is nonempty and bounded.
\end{lemma}

\begin{proof}
See Appendix~\ref{appendix: proofs in section main results}.
\end{proof}

\section{Empirical Studies} \label{section: Empirical Studies}  

In this section, we conduct empirical studies to analyze the effect of varying the Wasserstein ball's radius size~$\varepsilon$ and examine the impact of imposing different levels of transaction costs during rebalancing. Henceforth, we assume that the convex cost model follows a proportional cost structure; i.e., $TC_i(u_i(0)) := c_i w_i V(0)$, where~$c_i = c \in [0, 1)$ for all $i=1,\dots,m$. This cost model is indeed typical in real-world trading.\footnote{For example, trading on the Taiwan Stock Exchange typically incurs a transaction fee of $\alpha \cdot 0.1425\%$ of the trade value on buying, for some $\alpha \in (0, 1]$, and an additional transaction fee of $0.3\%$ of the trade value on selling. As another example, using professional broker services such as Interactive Brokers Pro may incur a cost of $\$0.005$ per share, with a minimum fee of $\$1$ and a maximum fee of $1\%$ of the trade value.} 

\medskip
\emph{Data.}
We consider a portfolio comprising a risk-free asset\footnote{For the risk-free rate, we use the Market Yield on U.S. Treasury Securities at One-Month Constant Maturity, Quoted on an Investment Basis. The daily data represents the yield on short-term U.S. Treasury bills, as provided by the Federal Reserve Board. The yearly average risk-free rate in 2022 and~2023 were about $1.67\%$ and $5.14\%$, respectively.} and risky assets from stocks ranked in the top~10 by market capitalization in the S\&P 500 at~2022,\footnote{See~\cite{spglobal2023}.} as shown in Table~\ref{table:S&P500's top 10 stocks}. 
The corresponding stock prices, obtained from Yahoo Finance for the two-year period from January~1, 2022, to January~1, 2024,  are shown in Figure~\ref{fig:stock_prices}. 
Note that 2022 was a bearish market, and 2023 was a bullish market, indicating a market regime switching. However, as demonstrated in later subsections, our proposed cost-sensitive DRO approach provides a robust solution to handle such market regime switches.

\begin{table}[h!]
%\centering
\caption{S\&P 500's Top 10 Stocks in 2022 Year-End}
\label{table:S&P500's top 10 stocks}
\resizebox{\columnwidth}{!}{%
	\begin{tabular}{ l l l  c}
		\toprule
		Rank & Company & Ticker & Percentage of  Total\\ & & &  Index Market Value (\%)\\
		\midrule
		1 & Apple Inc. & \texttt{AAPL} & 6.2\\
		
		2 & Microsoft Corporation & \texttt{MSFT} & 5.3 \\
		
		3 & Amazon.com Inc. & \texttt{AMZN} & 2.6 \\
		
		4 & Alphabet Inc. Class C & \texttt{GOOG} & 1.6 \\
		
		5 & Alphabet Inc. Class A & \texttt{GOOGL} & 1.6 \\
		
		6 & United Health Group Inc. & \texttt{UNH} & 1.5 \\
		
		7 & Johnson \& Johnson & \texttt{JNJ} & 1.4 \\
		
		8 & Exxon Mobil Corporation & \texttt{XOM} & 1.4 \\
		
		9 & Berkshire Hathaway Inc. Class B & \texttt{BRK.B} & 1.2 \\
		
		10 & JPMorgan Chase \& Co. & \texttt{JPM} & 1.3 \\
		\bottomrule
	\end{tabular}
}
\end{table}

%歷史日股價
\begin{figure}[h!]
\centering
\includegraphics[width=1\linewidth]{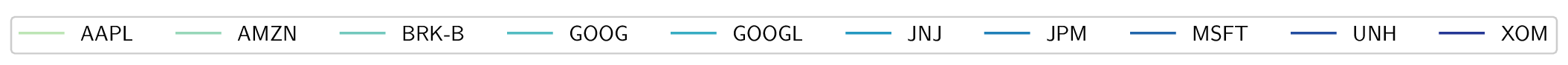}
\includegraphics[width=1\linewidth]{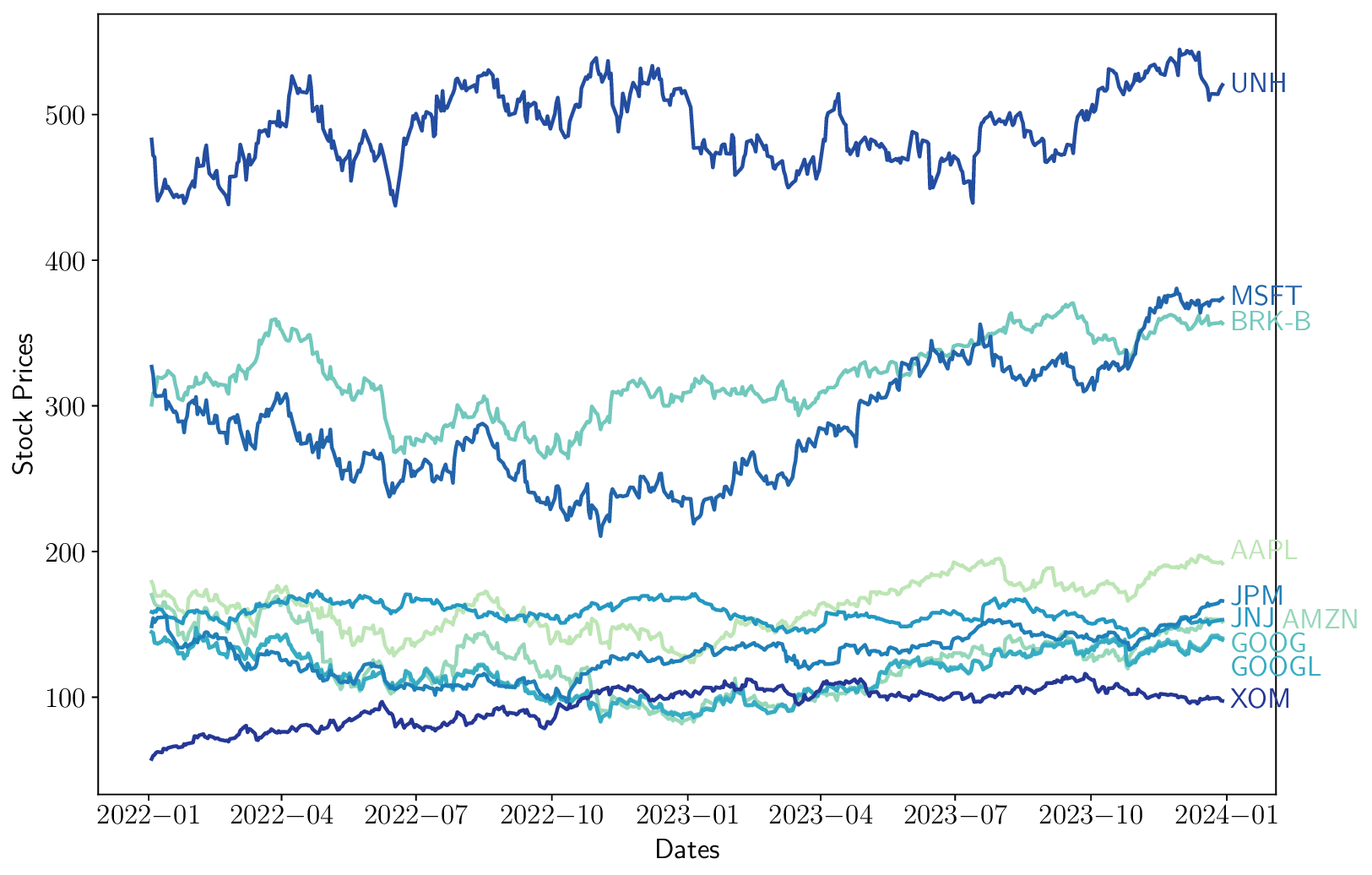}
\caption{Historical Adjusted Closing Prices for the 10 Selected Stocks of S\&P 500.}
\label{fig:stock_prices}
\end{figure}

% %無風險報酬
% \begin{figure}[h!]
%     \centering
%     \includegraphics[width=1\linewidth]{figs_two_year/risk_free_return.eps}
%     \caption{Daily Risk-Free Return.}
% \label{fig:risk-free return}
% \end{figure}

\emph{Simulation Details.}
Throughout this section, we use a \emph{sliding window approach} to repeatedly solve the proposed cost-sensitive DRO log-optimal problem during monthly rebalancing, similar to the approach in \cite{blanchet2022distributionally, wang2022data, hsieh2023solving}. This approach yields time-varying log-optimal weights that better reflect the dynamic market behavior. Starting with an initial account value of $V(0)=\$1$, we rebalance the portfolio using a one-month window size to obtain the latest optimal weight $w$.

Specifically, in the first month, we collected historical daily price data without making trades. This data was then used to generate initial random samples of asset returns. One common approach is to model the price dynamics using Geometric Brownian Motion (GBM) within a Monte Carlo simulation; see~\cite{gilli2019numerical}. By conducting 1,000 simulations, we generated a set of \emph{Monte Carlo-based} random samples of the finite random vector ${\mathcal{X}}_n \sim \mathbb{F}$ for the first month. These samples were subsequently used to calculate the optimal weight applied throughout the second month. This process was repeated each month by sliding the window of historical data and resolving the problem until the end of 2023. Notably, transaction costs were applied only to the risky assets during rebalancing periods, while \emph{no} costs were incurred for the risk-free asset.

\subsection{Effect of Ambiguity Sizes}
This section examines the effect of the ambiguity size of the Wasserstein ball, in terms of $\varepsilon$, under various transaction cost levels.  
We solve the cost-sensitive DRO log-optimal problem using a sliding window approach for different values of the radius~$\varepsilon \in [0, 2]$ with a window size of one month. This yields twelve sets of optimal weights~$w$ per year.  Figures~\ref{fig:Varying Radius Size during rebalancing 2022} and~\ref{fig:Varying Radius Size during rebalancing 2023} depict the resulting optimal weights for the first month of 2022 and 2023, respectively. 
Similar patterns are observed in the weights obtained for the remaining months. 
For example, see Figures~\ref{fig:twelve figure with 0.1 transaction costs during rebalancing 2022} and~\ref{fig:twelve figure with 0.1 transaction costs during rebalancing 2023} for the set of optimal weight across all twelve months under~0.1\% costs in 2022 and~2023, respectively.

In the absence of transaction costs, increasing the radius size leads to an equal-weighted portfolio, i.e., $w_i = 1/m$ for all $i$, indicating greater portfolio diversification. 
When transaction costs are considered, as shown in Figures~\ref{fig:twelve figure with 0.1 transaction costs during rebalancing 2022} and \ref{fig:twelve figure with 0.1 transaction costs during rebalancing 2023}, the portfolio weights initially tend to an equal-weighted portfolio, which is consistent with \cite{mohajerin2018data, li2023wasserstein}. However, as the radius of the Wasserstein ball increases further, the portfolio shifts slightly towards the risk-free asset.
Comprehensive optimal weights studies can be found in the supplementary document.

%調整半徑大小權重示意圖
%2022
\begin{figure}[h!]
\centering
\includegraphics[width=.9\linewidth]{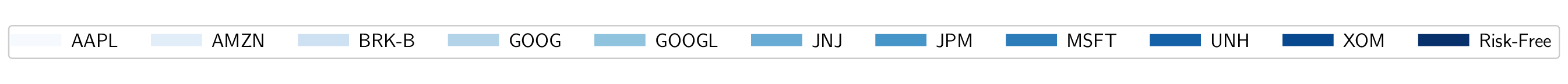}
\includegraphics[width=.9\linewidth]{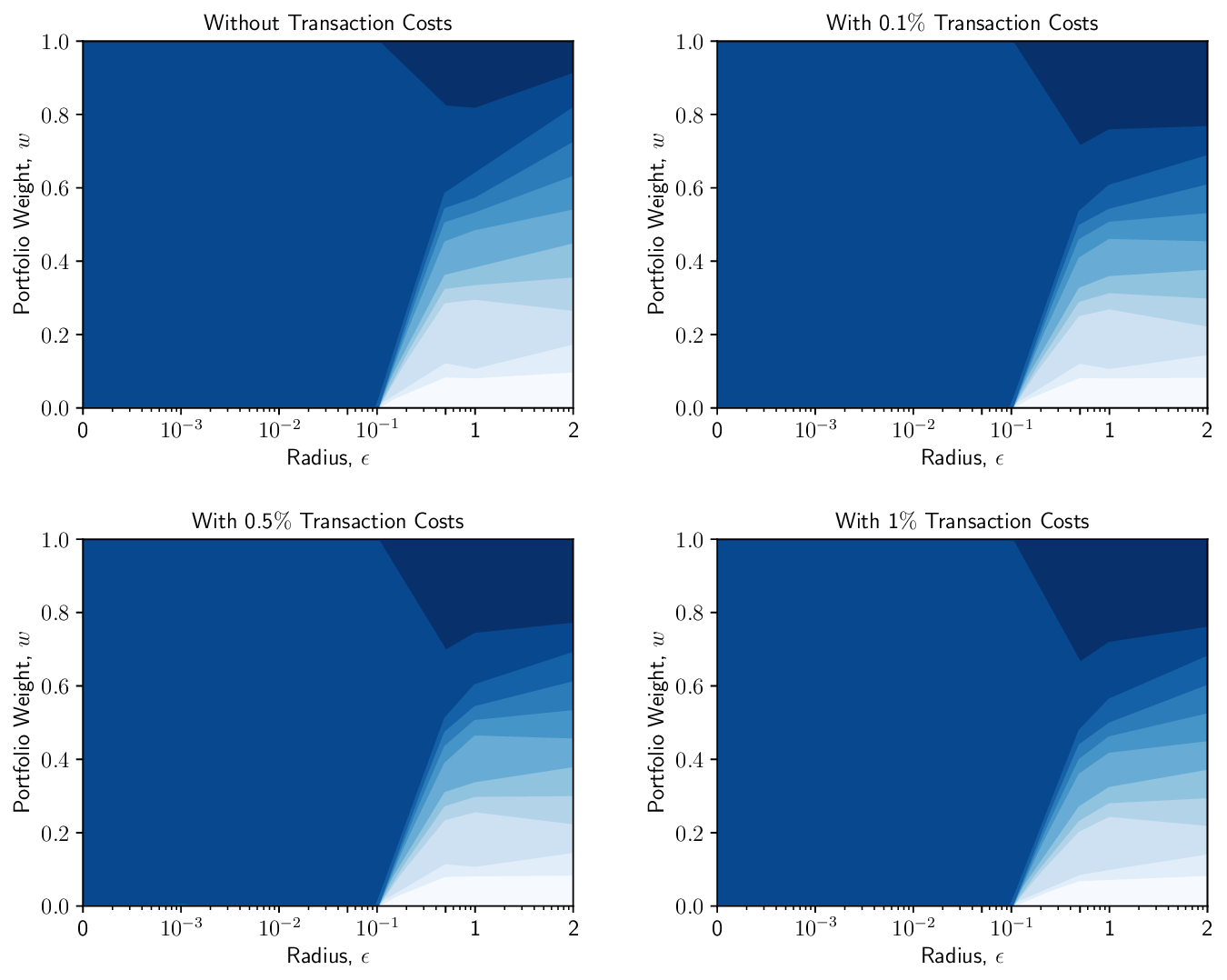}
\caption{Effect of Varying Ambiguity Size on Optimal Weights Using Data from January of~2022.}   
\label{fig:Varying Radius Size during rebalancing 2022}
\end{figure}

%2023
\begin{figure}[h!]
\centering
\includegraphics[width=.9\linewidth]{figs_two_year/legend_only.eps}
\includegraphics[width=.9\linewidth]{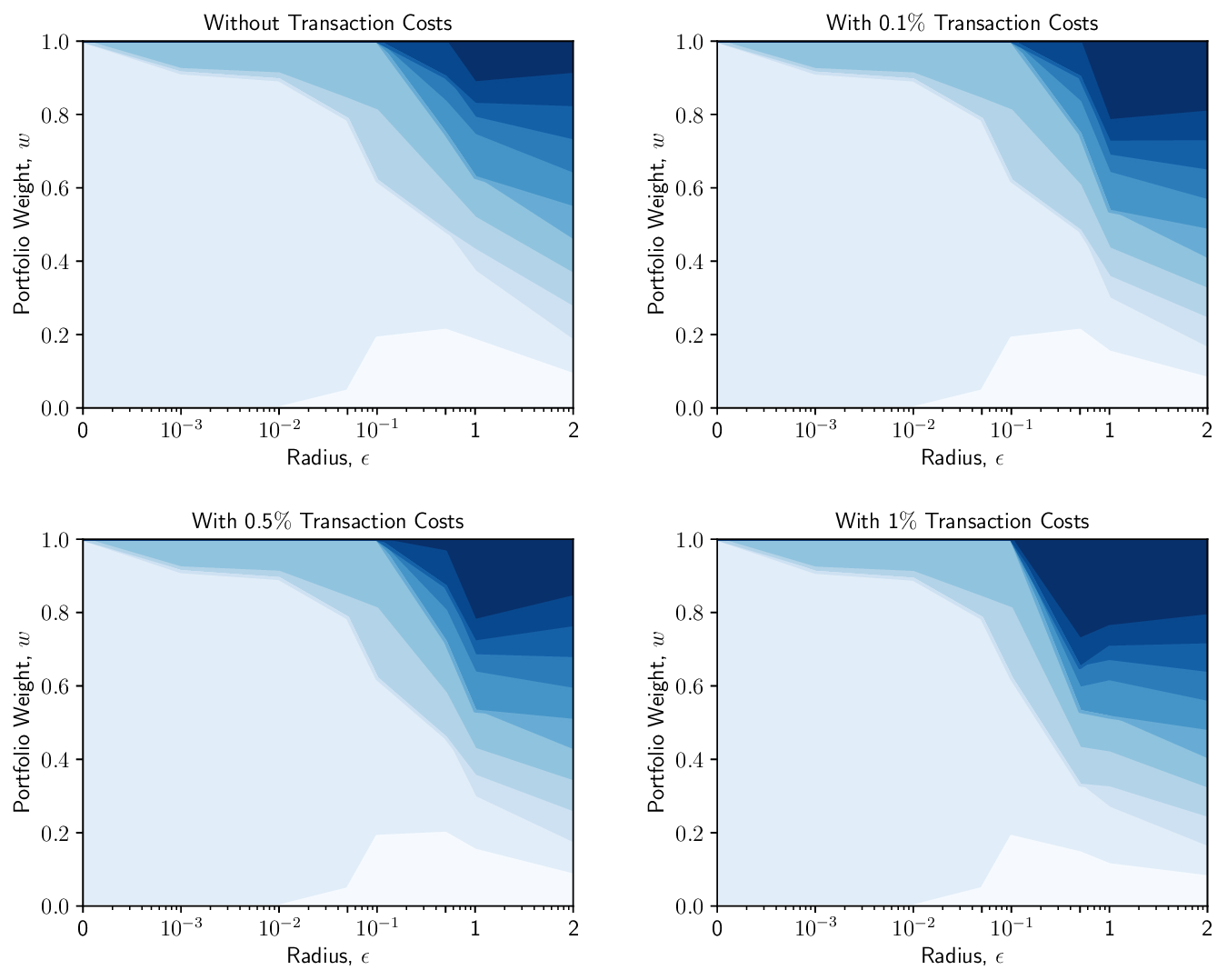}
\caption{Effect of  Varying Ambiguity Size on Optimal Weights Using Data from January of~2023.}   
\label{fig:Varying Radius Size during rebalancing 2023}
\end{figure}

\subsection{Out-Of-Sample Trading Performance}
 Figure~\ref{fig:Account Value Trajectories with Different Transaction Costs} shows the out-of-sample trading performance, in terms of the account value trajectories, for fixed transaction costs and varying Wasserstein radius sizes~$\varepsilon \in \{0, 10^{-3}, 10^{-2}, 10^{-1}, 1\}$. For comparison, a benchmark path tracking the S\&P 500 ETF (Ticker:~\texttt{SPY}) is also included, with no cost applied to this path. 
 
 With the monthly rebalancing approach, account values only change at the beginning of each month, as the investor makes no trades between rebalancing periods. The figure shows that, during the 2022 period, the portfolio with $\varepsilon = 0$, which fully trusts the historical data, performs better, regardless of the transaction cost level. However, by the end of the 2023 year, the portfolio with~$\varepsilon = 0$ exhibits the worst overall performance. In contrast, as $\varepsilon$ increases, the extent of the decline in portfolio value diminishes and, in several cases, outperforms the S\&P 500 benchmark by late 2023. This pattern suggests that a larger radius leads to more diversified portfolio allocations, helping to mitigate potential downside~risks.

\begin{figure}[h!]
\centering
\includegraphics[width=1\linewidth]{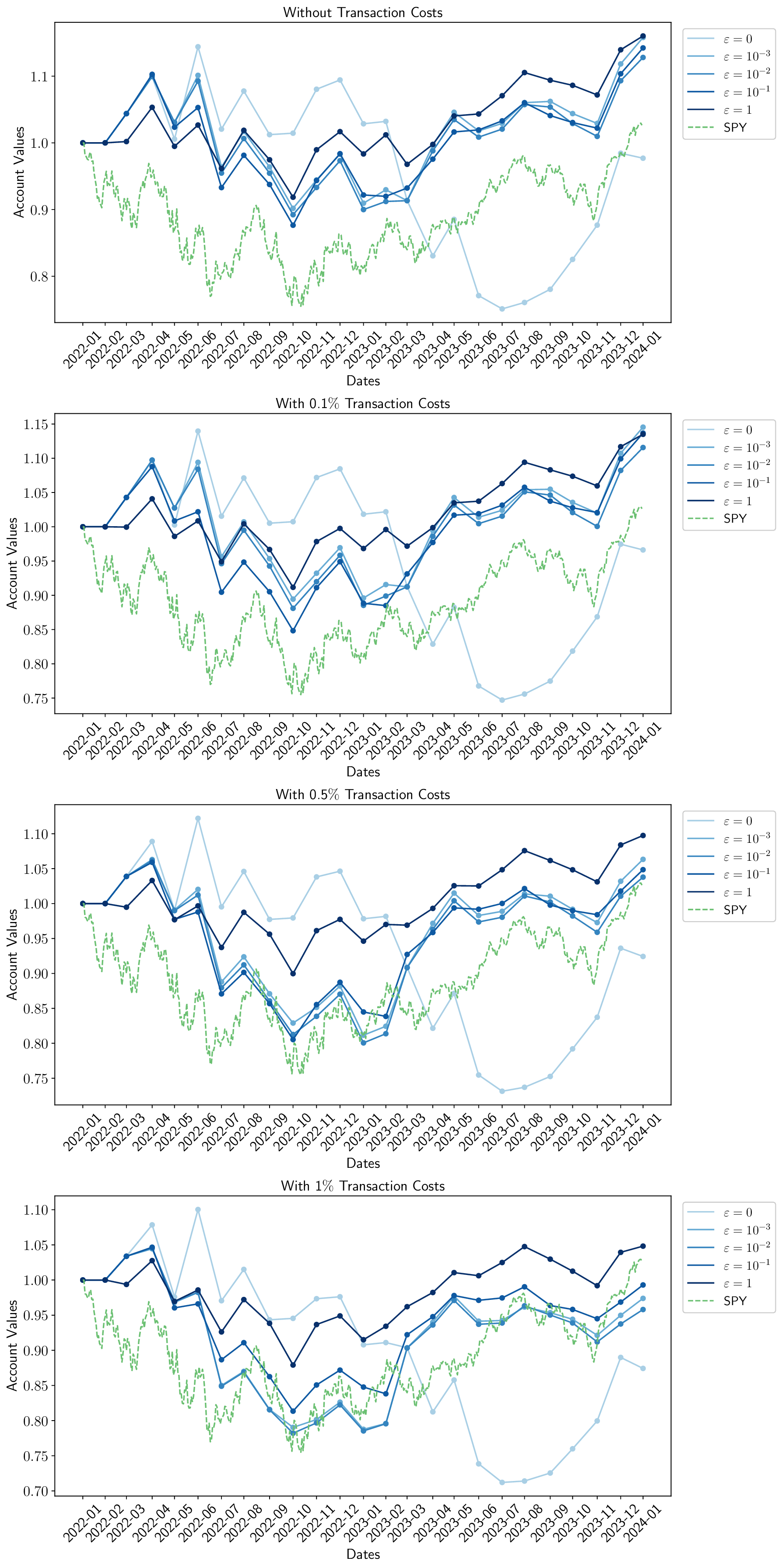}
\caption{Out-of-Sample Monthly Rebalanced Account Value Trajectories: Different Radius Sizes and Transaction Costs. Note that $\varepsilon = 0$ Corresponds to Classical ELG Strategy, and $\varepsilon > 0$ Corresponds to Robust ELG Strategy with $\varepsilon = 1$ Approximating to an Equal Weights Portfolio. }   
\label{fig:Account Value Trajectories with Different Transaction Costs}
\end{figure}

Tables~\ref{table:Without Transaction Costs during rebalancing}--\ref{table:With 1 Transaction Costs during rebalancing} summarize the trading performance, including cumulative return (\texttt{CR}), annualized standard deviation (\texttt{STD}), annualized Sharpe ratio (\texttt{SR}), and maximum drawdown (\texttt{MDD}).\footnote{With the $\{V(k)\}_{k=0}^{N}$ be the realized account value sequence from stage $k=0$ to $k=N$, the cumulative return \texttt{CR} is defined as $\frac{V(N) - V(0)}{ V(0)}$. The per-period return is defined as $R(k):= \frac{V(k+1) - V(k)}{ V(k)}$ and $\overline{R}:= \frac{1}{N}\sum_{k=0}^{N} R(k)$ is the sample mean of $R(k)$. The (sample) standard deviation \texttt{STD} is defined as $\sqrt{N} \cdot \texttt{STD}_{\rm daily}$ with $\texttt{STD}_{\rm daily} := \sqrt{ \frac{1}{N-1} \sum_{k=0}^{N-1} (R(k) - \overline{R})^2} $. The Sharpe ratio \texttt{SR} is defined as \texttt{SR}:= $\sqrt{N} \frac{\overline{R} - r_f}{\texttt{STD}}$. The maximum percentage drawdown \texttt{MDD}:= $\max_{1 \leq k \leq N} \frac{V_{\max}(k) - V(k)}{V_{\max}(k)}$, where $V_{\max}(k) := \max_{0 \leq l \leq k} V(l)$.} 
As observed from the tables, when~$\varepsilon = 1$, the portfolio demonstrates lower standard deviation (\texttt{STD}) and maximum drawdown (\texttt{MDD}), along with higher cumulative return (\texttt{CR}) and annualized Sharpe ratio (\texttt{SR}) compared to when $\varepsilon=0$. This indicates that larger radius size $\varepsilon$ may enhance the robustness of out-of-sample performance, leading to less volatile outcomes compared to scenarios without ambiguity consideration. Moreover, these performance improvements are not limited to $\varepsilon = 1$; in fact, whenever~$\varepsilon \neq 0$, i.e., when ambiguity is considered, the portfolio consistently outperforms both \texttt{SPY} and the $\varepsilon =0$~case.

% Performance Table

\begin{table}[h!]
	\centering
	\caption{Trading Performance Without Transaction Costs}
	\label{table:Without Transaction Costs during rebalancing}
	\begin{tabular}{l cccc}
		\toprule
		$\varepsilon$ & \texttt{CR} & \texttt{STD} & \texttt{SR} & \texttt{MDD} \\
		\midrule
		$0$       &0.977 & 0.260 & -0.052 & 34.35$\%$ \\
		$10^{-3}$ &1.158 & 0.195 & 0.315 & 18.11$\%$ \\
		$10^{-2}$ &1.128 & 0.191 & 0.246 & 18.86$\%$ \\
		$10^{-1}$ &1.142 & 0.176 & 0.289 & 20.52$\%$ \\
		$1$       &1.160 & 0.144 & 0.374 & 12.79$\%$\\
        \midrule
        \texttt{SPY} & 1.026 & 0.195 & -0.006 & 24.50$\%$ \\
		\bottomrule
	\end{tabular}
\end{table}

\begin{table}[h!]
	\centering
	\caption{Trading Performance With $0.1\%$ Transaction Costs}
	\label{table:With 0.1 Transaction Costs during rebalancing}
	\begin{tabular}{l cccc}
		\toprule
		$\varepsilon$ & \texttt{CR} & \texttt{STD} & \texttt{SR} & \texttt{MDD} \\
		\midrule
		$0$       &0.966 & 0.258 & -0.077 & 34.44$\%$ \\
		$10^{-3}$ &1.145 & 0.192 & 0.287 & 18.49$\%$ \\
		$10^{-2}$ &1.116 & 0.190 & 0.216 & 19.71$\%$ \\
		$10^{-1}$ &1.137 & 0.175 & 0.275 & 22.02$\%$ \\
		$1$       &1.135 & 0.129 & 0.312 & 12.41$\%$\\

		\bottomrule
	\end{tabular}
\end{table}

\begin{table}[h!]
	\small
	\centering
	\caption{Trading Performance With $0.5\%$ Transaction Costs}
	\label{table:With 0.5 Transaction Costs during rebalancing}
	\begin{tabular}{l cccc}
		\toprule
		$\varepsilon$ & \texttt{CR} & \texttt{STD} & \texttt{SR} & \texttt{MDD} \\
		\midrule
		$0$       &0.924 & 0.251 & -0.178 & 34.82$\%$ \\
		$10^{-3}$ &1.063 & 0.187 & 0.082 & 23.65$\%$ \\
		$10^{-2}$ &1.038 & 0.190 & 0.017 & 24.61$\%$ \\
		$10^{-1}$ &1.049 & 0.171 & 0.030 & 23.96$\%$ \\
		$1$       &1.097 & 0.124 & 0.178 & 12.91$\%$\\
		\bottomrule
	\end{tabular}
\end{table}

\begin{table}[h!]
	\small
	\centering
	\caption{Trading Performance With $1\%$ Transaction Costs}
	\label{table:With 1 Transaction Costs during rebalancing}
	\begin{tabular}{l cccc}
		\toprule
		$\varepsilon$ & \texttt{CR} & \texttt{STD} & \texttt{SR} & \texttt{MDD} \\
		\midrule
		$0$       &0.874 & 0.242 & -0.312 & 35.29$\%$ \\
		$10^{-3}$ &0.974 & 0.180 & -0.178 & 24.59$\%$ \\
		$10^{-2}$ &0.958 & 0.181 & -0.222 & 25.19$\%$ \\
		$10^{-1}$ &0.993 & 0.147 & -0.185 & 22.28$\%$ \\
		$1$       &1.048 & 0.124 & -0.015 & 14.46$\%$\\
		\bottomrule
	\end{tabular}
\end{table}

\subsection{Effects of Transaction Costs During Rebalancing}
This section examines the effects of different transaction cost levels during rebalancing while maintaining the radius fixed. Specifically, we impose proportional transaction costs of $c \in \{0\%$, $0.1\%$, $0.5\%$, $1\%\}$ to portfolios with fixed radius sizes $\varepsilon \in \{0, 10^{-3}, 10^{-2}, 10^{-1}, 1, 2 \}$. Figure~\ref{fig:Account Value Trajectories with Different Radius Sizes} illustrates that the cumulative return~(\texttt{CR}) decreases as transaction costs increase, leading to a deterioration in trading performance.

%Varying transaction costs trajectories
\begin{figure}[h!]
\centering
\includegraphics[width=1\linewidth]{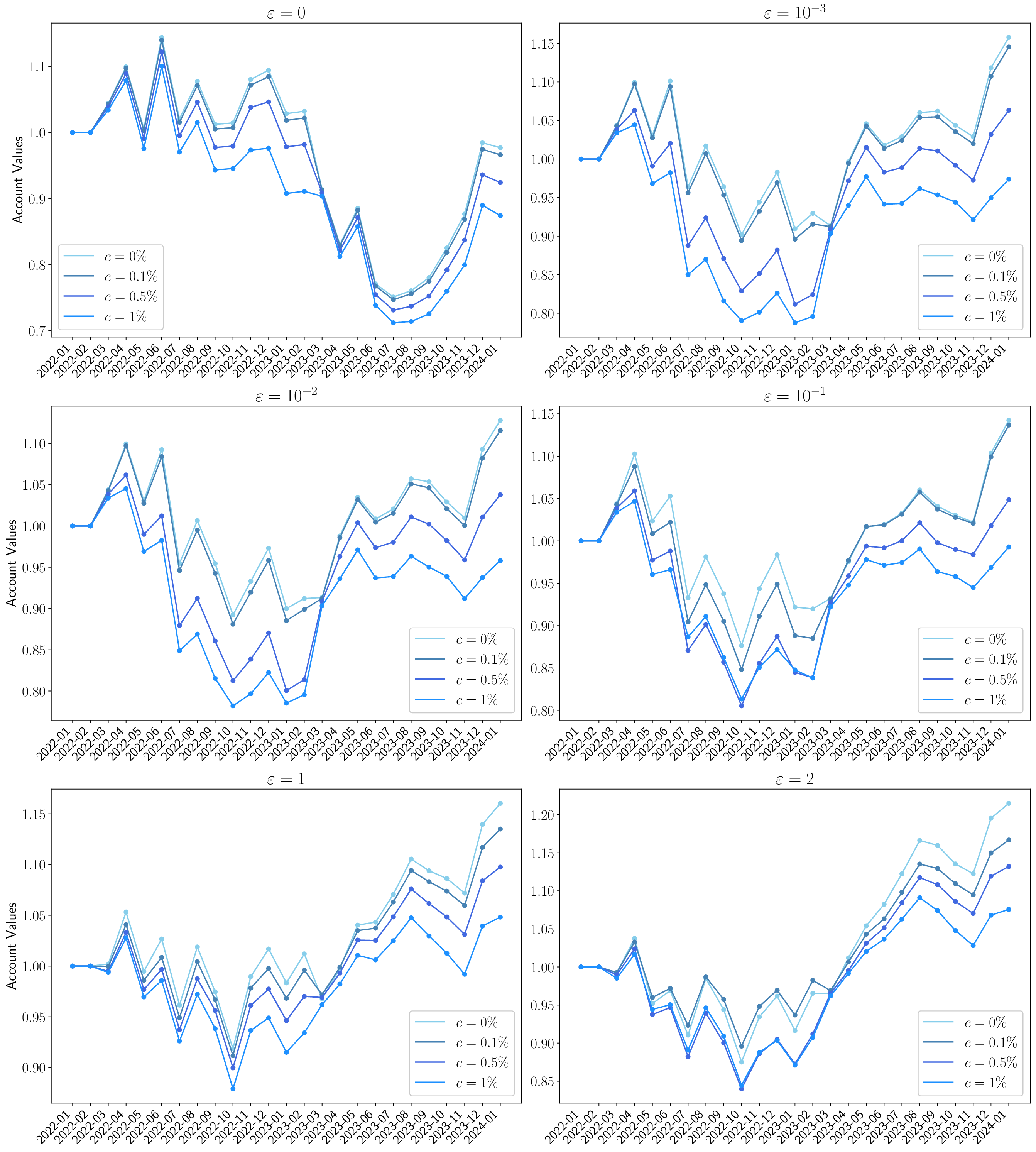}
\caption{Transaction Cost Effect: Out-of-Sample Monthly Rebalanced Account Value Trajectories.}   
\label{fig:Account Value Trajectories with Different Radius Sizes}
\end{figure}

%2022
%c=0(3*4), 2022			
%\begin{figure}[h!]
%	\centering
%	\includegraphics[width=1\linewidth]{figs/legend_only.eps}
%	\includegraphics[width=\linewidth]{figs_two_year/w_results_MC_c0_100runs_2022.eps}
%	\caption{Without Transaction Costs in 2022.}
%	\label{fig:twelve figure without transaction costs during rebalancing 2022}
%\end{figure}	

%c=0.1%(3*4), 2022		
\begin{figure}[h!]
\centering
\includegraphics[width=1\linewidth]{figs_two_year/legend_only.eps}
\includegraphics[width=\linewidth]{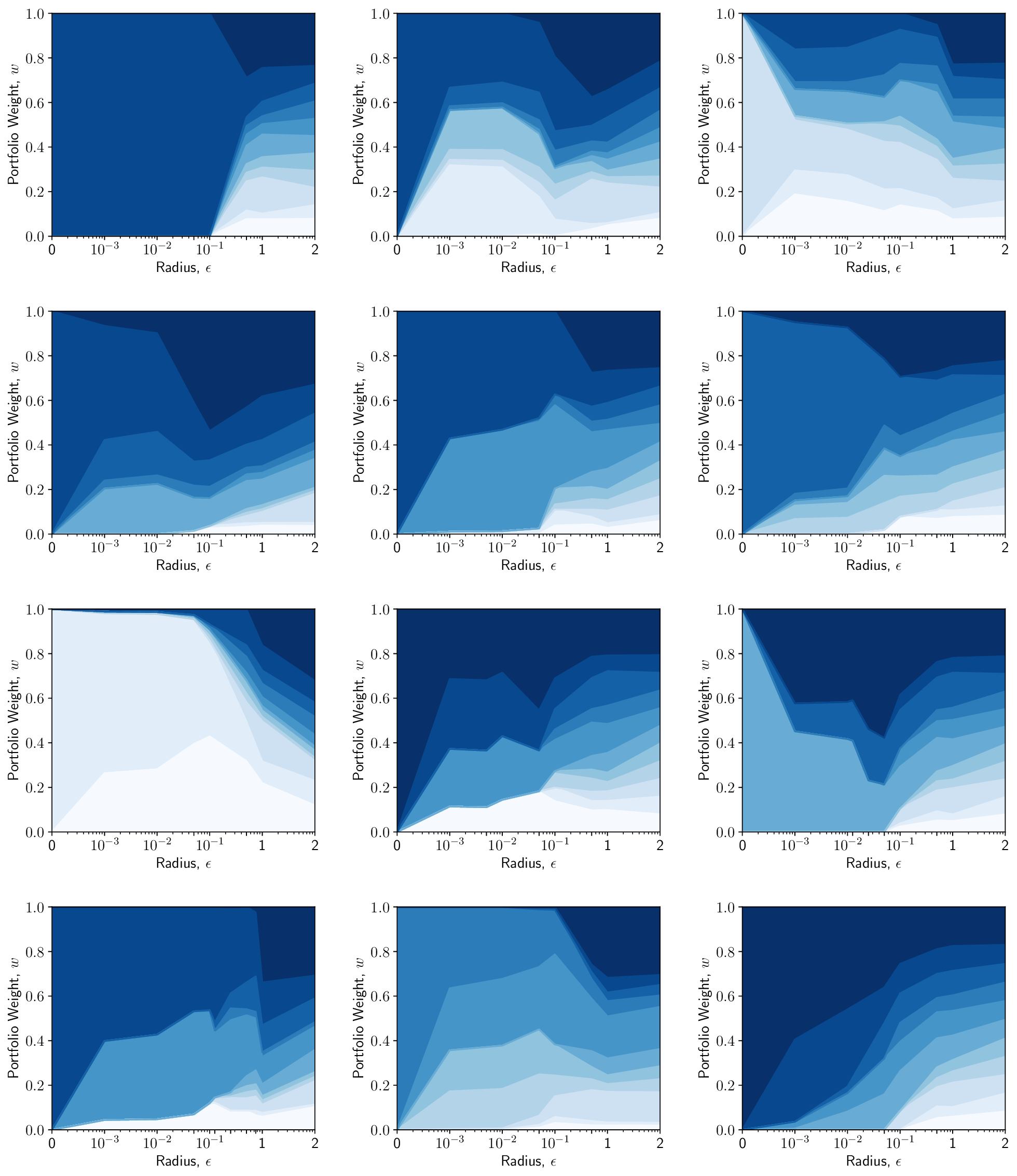}
\caption{Monthly Optimal Weight in 2022:  With $0.1\%$ Transaction Costs.}
\label{fig:twelve figure with 0.1 transaction costs during rebalancing 2022}
\end{figure}

%c=0.5%(3*4), 2022
%\begin{figure}[h!]
%	\centering
%	\includegraphics[width=1\linewidth]{figs/legend_only.eps}
%	\includegraphics[width=\linewidth]{figs_two_year/w_results_MC_c05_100runs.eps}
%	\caption{With $0.5\%$ Transaction Costs During Rebalancing in 2023.}
%	\label{fig:twelve figure with 0.5 transaction costs during rebalancing 2022}
%\end{figure}

%2023
%c=0(3*4), 2023			
%\begin{figure}[h!]
%	\centering
%	\includegraphics[width=1\linewidth]{figs/legend_only.eps}
%	\includegraphics[width=\linewidth]{figs_two_year/w_results_MC_c0_100runs.eps}
%	\caption{Without Transaction Costs.}
%	\label{fig:twelve figure without transaction costs during rebalancing}
%\end{figure}	

%c=0.1%(3*4), 2023					
\begin{figure}[h!]
\centering
\includegraphics[width=1\linewidth]{figs_two_year/legend_only.eps}
\includegraphics[width=\linewidth]{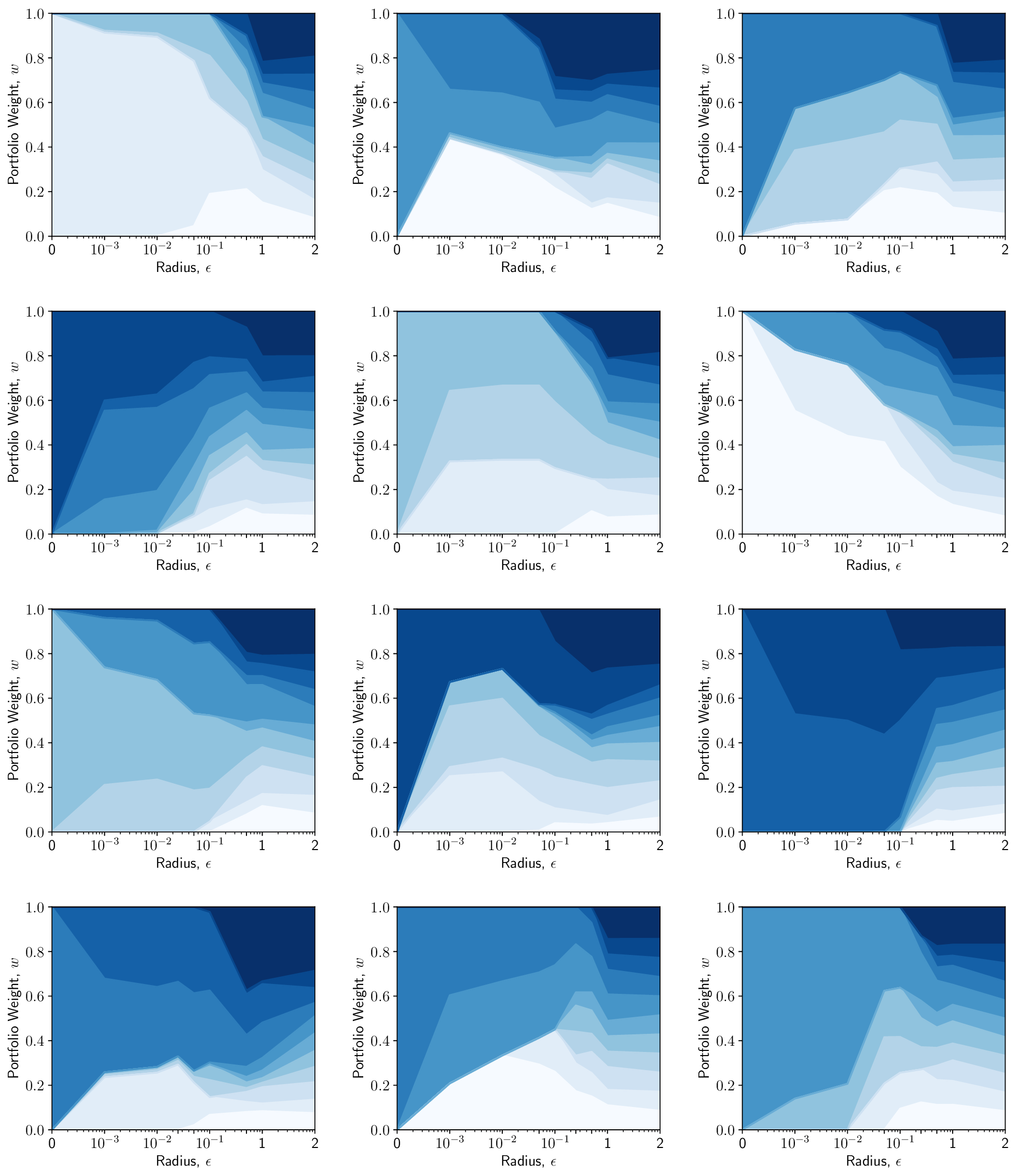}
\caption{Monthly Optimal Weight in 2023:  With $0.1\%$ Transaction Costs.}
\label{fig:twelve figure with 0.1 transaction costs during rebalancing 2023}
\end{figure}

%c=0.5%(3*4), 2023			
%\begin{figure}[h!]
%	\centering
%	\includegraphics[width=1\linewidth]{figs/legend_only.eps}
%	\includegraphics[width=\linewidth]{figs_two_year/w_results_MC_c05_100runs.eps}
%	\caption{With $0.5\%$ Transaction Costs During Rebalancing.}
%	\label{fig:twelve figure with 0.5 transaction costs during rebalancing}
%\end{figure}

%c=1%(3*4), 2023			
%\begin{figure}[h!]
%	\centering
%	\includegraphics[width=1\linewidth]{figs/legend_only.eps}
%	\includegraphics[width=\linewidth]{figs_two_year/w_results_MC_c1_100runs.eps}
%	
%	\caption{With $1\%$ Transaction Costs During Rebalancing.}
%	\label{fig:twelve figure with 1 transaction costs during rebalancing}
%\end{figure}

%
%{
%\color{red}
%\subsection{More Assets and Other Type of Convex Costs (Piecewise Linear or Nonlinear Costs)}
%}

\section{Conclusion}
In this paper, we introduced a novel cost-sensitive distributionally robust log-optimal portfolio framework using the Wasserstein metric alongside a general convex transaction cost model. We established conditions for robustly survivable trades across distributions within the Wasserstein ball and approximated the infinite-dimensional distributionally robust optimization problem as a finite convex program.

Our empirical studies show that, in the absence of transaction costs, the portfolio tends to converge toward an equal-weighted allocation. With transaction costs, the portfolio shifts toward slightly the risk-free asset, balancing the trade-off between costs and optimal asset allocation. Notably, the portfolio proposed exhibited lower volatility and maximum drawdown compared to strategies without ambiguity considerations, highlighting the practical value of our approach.

Future work could explore extending this framework to accommodate larger asset sets, addressing the computational challenges inherent in such cases. Efficient algorithms, such as those in \cite{hsieh2023solving}, would be crucial for scalability. Additionally, characterizing the optimal radius of the ambiguity set, where this transition occurs, would be of significant interest, as discussed in \cite{blanchet2019robust}.

\bibliographystyle{apalike}
\bibliography{refs}

\appendix
\section{Technical Proofs} \label{appendix: technical proofs}
This appendix collects some technical proofs in Sections~\ref{section: Problem Formulation} and \ref{section: Main Results}.

\subsection{Proofs in Section~\ref{section: Problem Formulation}} \label{appendix: proofs in section formulation}

\begin{proof}[Proof of Lemma~\ref{lemma: robust survival condition with transaction costs rebalancing}]
Fix $ n \geq 1 $ and $ \varepsilon \geq 0.$
We begin by noting that all distributions $ \mathbb{F} \in \mathcal{B}_\varepsilon( \widehat{\mathbb{F}} )  \subseteq \mathcal{M}(\mathfrak{X}) $ are supported on~$ {\mathfrak{X}} $. Thus, according to~\eqref{eq: support set of compound return},  the possible values of $ {\mathcal{X}}_n $ are bounded between $ {\mathcal{X}}_{\min,i} $ and $ {\mathcal{X}}_{\max,i} $ for each asset~$ i $.
Observe that
\begin{align*}
	V(n) 
	& = V(0) \left( c(w) + w^\top {\mathcal{X}}_n \right) \\
	&\geq V(0) \left( c(w) + \sum_{i=1}^{m} w_i {\mathcal{X}}_{\min,i} \right)
	:= V_{\min} (n)
\end{align*}
Since 	$c(w)+	\sum_{i=1}^{m} w_i {\mathcal{X}}_{\min, i} > 0$, it follows immediately that~$ V_{\min}(n) > 0 $.
Additionally, note that~$ V(n) \geq V_{\text{min}}(n) > 0 $ holds almost surely under any~$ \mathbb{F} \in \mathcal{B}_\varepsilon( \widehat{\mathbb{F}} ) $. Hence, it follows~that
$
\mathbb{P}^\mathbb{F}( V(n) > 0 )  	= \mathbb{P}^\mathbb{F}(  c(w) + w^\top {\mathcal{X}}_n > 0 ) = 1
$ for all $n \geq 1$,
and the proof is complete.
\end{proof}

\subsection{Proofs in Section~\ref{section: Main Results}} \label{appendix: proofs in section main results}

\begin{proof}[Proof of Lemma~\ref{lemma: Decomposition of Marginal Distribution}]
We begin by	defining the events $B_j :=\{\mathcal{X}' = \widehat{\mathcal{X}}_j \}$ for $j=1, 2, \dots,N$.
Since the empirical distribution $\widehat{ \mathbb{F} } = \frac{1}{N} \sum_{j=1}^N \delta_{\widehat{\mathcal{X}}_j}$, each event $B_j$ occurs with an equal probability: $\mathbb{P}(B_j) := \mathbb{P}(\mathcal{X}' = \widehat{\mathcal{X}}_j) = 1/N.$ Moreover, by definition of the conditional distribution, for any measurable set $A$, we have $\mathbb{P}( \mathcal{X} \in A | B_j) = \mathbb{F}_j(A)$. Applying the Law of Total Probability, we obtain
\begin{align*}
	\mathbb{F}(A) 
	&=  \mathbb{P}(\mathcal{X} \in A) \\
	& = \sum_{j=1}^N \mathbb{P}( \mathcal{X} \in A | B_j) \mathbb{P}(B_j) 
	= \sum_{j=1}^N \mathbb{F}_j(A) \frac{1}{N}
\end{align*}
Since the equality holds for any measurable set $A$, it follows that
$
\mathbb{F} = \frac{1}{N}\sum_{j=1}^N \mathbb{F}_j(A)
$ and the proof of part~$(i)$ is complete.

To prove part~$(ii)$, we show the joint distribution $\Pi$ can be constructed using the Law of Total Probability.
For any measurable sets $A, B_j \subseteq \mathcal{ { \mathfrak{X} } }$. Since $\widehat{ \mathbb{F} }$ is discrete and $\mathcal{X}$ takes values  $\widehat{\mathcal{X}}_j$ with equal probability, we can partition the event $\{\mathcal{X}' \in B\}$ into disjoint events based on each~$\widehat{\mathcal{X}}_j$: 
$
\{\mathcal{X}' \in B\} = \cup_{j=1}^N \{\mathcal{X}' = \widehat{\mathcal{X}}_j\}\cap B.
$
Applying the Law of Total Probability yields
\begin{align*}
	\Pi(A\times B) 
	&= \mathbb{P}(\mathcal{X} \in A, \mathcal{X}' \in B)\\
	&= \sum_{j=1}^N \mathbb{P}(\mathcal{X} \in A| \mathcal{X}' = \widehat{\mathcal{X}}_j ) \mathbb{P}(\mathcal{X}' = \widehat{\mathcal{X}}_j) \mathbb{1}_{ \{ \widehat{\mathcal{X}}_j \in B\} }\\
	&=\frac{1}{N} \sum_{j=1}^N \mathbb{F}_j(A)  \mathbb{1}_{ \{ \widehat{\mathcal{X}}_j \in B\} }
\end{align*}
where $\mathbb{1}_{ \{ \widehat{\mathcal{X}}_j \in B \}}= 1$ if $\widehat{\mathcal{X}}_j \in B$ and $0$ otherwise.
Recalling that the product measure $\delta_{\widehat{\mathcal{X}}_j} \otimes \mathbb{F}_j$ evaluated on $A \times B$ is
\[
( \delta_{\widehat{\mathcal{X}}_j}  \otimes \mathbb{F}_j ) (A \times B) = \mathbb{F}_j(A) \cdot \delta_{\widehat{\mathcal{X}}_j}(B) =  \mathbb{F}_j(A)  \cdot \mathbb{1}_{\{ \widehat{\mathcal{X}}_j \in B\} }.
\]
Hence, it follows that
$
\Pi(A\times B) = \frac{1}{N} \sum_{j=1}^N ( \delta_{\widehat{\mathcal{X}}_j}  \otimes \mathbb{F}_j ) (A \times B).
$
Since this equality holds for all measurable rectangles~$A \times B \subseteq \mathcal{ { \mathfrak{X} } } \times \mathcal{ { \mathfrak{X} } }$, and the collection of such rectangles generates the product $\sigma$-algebra on $\mathcal{ { \mathfrak{X} } } \times \mathcal{ { \mathfrak{X} } }$, by the $\pi$-$\lambda$ Theorem, the equality extends to all measurable sets in the product space. Therefore, the joint distribution $\Pi$ can be expressed as
$
\Pi = \frac{1}{N}\sum_{j=1}^{N}\delta_{\widehat{\mathcal{X}}_j} \otimes \mathbb{F}_j. 
$
\end{proof}

\begin{proof}[Proof of Lemma~\ref{lemma: equity before auxilary variables with transaction costs during rebalancing}]
Define a shorthand notation: $\phi( {\mathcal{X}}, \widehat{\mathcal{X}}_j) :=  \log \left(1 + w^\top  {\mathcal{X}}\right) + \lambda \| {\mathcal{X}} - \widehat{\mathcal{X}}_j\|$. We must show that 
$
\inf_{\mathbb{F}_j \in \mathcal{M}({ \mathfrak{X} })}   \int \phi ({\mathcal{X}}, \widehat{\mathcal{X}}_j)d{\mathbb{F}_j}
=
\inf_{ {\mathcal{X}} \in  { \mathfrak{X} }} \phi ( {\mathcal{X}}, \widehat{\mathcal{X}}_j).
$
Observe that the left-hand side can be written as
\begin{align}
	\inf_{\mathbb{F}_j\in \mathcal{M}( { \mathfrak{X} })}   \int \phi( {\mathcal{X}},  \widehat{\mathcal{X}}_j) d{\mathbb{F}_j}  
	&  \leq \inf_{ {\mathcal{X}} \in  { \mathfrak{X} }} \int \phi(x, \widehat{\mathcal{X}}_j) \delta_{ {\mathcal{X}}}(dx) \label{inequality:DRO to RO with transaction costs during rebalancing} \\
	&  =\inf_{ {\mathcal{X}} \in  { \mathfrak{X} }}  \phi( {\mathcal{X}}, \widehat{\mathcal{X}}_j) \notag
\end{align}
where inequality~\eqref{inequality:DRO to RO with transaction costs during rebalancing} follows from the fact that $\mathcal{M}({ \mathfrak{X} })$ contains all
the Dirac distributions supported on ${ \mathfrak{X} }$.
Next, we consider
\begin{align*}
	&	\inf_{\mathbb{F}_j\in \mathcal{M}( { \mathfrak{X} })} \int \phi( {\mathcal{X}}, \widehat{\mathcal{X}}_j)d{\mathbb{F}_j} \\
	&\qquad \geq \inf_{\mathbb{F}_j\in \mathcal{M}( { \mathfrak{X} })}  \int \inf_{ {\mathcal{X}} \in  {\mathfrak{X}}} \phi( {\mathcal{X}}, \widehat{\mathcal{X}}_j) d{\mathbb{F}_j} \\
	&\qquad = \inf_{\mathbb{F}_j\in \mathcal{M}( { \mathfrak{X} })}  \inf_{ {\mathcal{X}} \in  { \mathfrak{X} }} \phi( {\mathcal{X}}, \widehat{\mathcal{X}}_j) (\int d{\mathbb{F}_j}) \\
	%		&\qquad =  \inf_{ {\mathcal{X}} \in  { \mathfrak{X} }} \phi( {\mathcal{X}}, \widehat{\mathcal{X}}_j) \left( \inf_{\mathbb{F}_j \in \mathcal{M} ({ \mathfrak{X} }), \forall j} \int d{\mathbb{F}_j} \right) \\
	&\qquad =  \inf_{ {\mathcal{X}} \in  { \mathfrak{X} }} \phi( {\mathcal{X}}, \widehat{\mathcal{X}}_j)
\end{align*}
Hence, the proof is complete.
\end{proof}

\begin{proof}[Proof of Lemma~\ref{lemma:turn into inf_sup with transaction costs during rebalancing}]
%
%We begin by recalling that the dual problem is
%\begin{align*}
%	&d^*:=	\sup_{\lambda,  {s}_j, z_j}\, -\lambda\varepsilon  + \frac{1}{N} \sum_{j=1}^{N}  {s}_j \\
%	&	\text{s.t. }  \
%	\medmath{	\min_{ {\mathcal{X}} \in  { \mathfrak{X} }} \left[ \log \left( c(w) +  w^\top  {\mathcal{X}}\right) +  z_j^\top ( {\mathcal{X}} - \widehat{\mathcal{X}}_j) \right] \geq  {s}_j , \; j=1,\dots,N} \\
%	&	\qquad \lambda \geq 0 \\
%	&	\qquad {\|z_j\|_* \leq \lambda} , \quad j=1,\dots,N
%\end{align*}
%where $z_j \in \mathbb{R}^m$, $\| z_j \|_*$ is the dual norm of $z_j$. 

To show that the dual optimal value is finite, we first show that the dual problem is feasible, meaning there exists at least one set of variables $(\lambda, s_j, z_j)$ satisfying the constraints.
Specifically, consider choosing $\lambda = 0$ and $z_j = 0$ for all $ j = 1, \dots, N $. The constraints then reduce~to:
\begin{align} \label{eq: bounds for s_j}
	\min_{\mathcal{X} \in \mathfrak{X}} \left[ \log\left( c(w) + w^\top \mathcal{X} \right) \right] \geq s_j, \quad \forall j. 
\end{align}
Since $ c(w) + w^\top \mathcal{X} > 0 $ for all $ \mathcal{X} \in \mathfrak{X}$ due to the robust survivability Lemma~\ref{lemma: robust survival condition with transaction costs rebalancing}, the minimum exists and is finite. Therefore, setting $s_j$ to this finite minimum value satisfies the constraints. This shows that the dual problem is feasible.

To see that the dual optimal value is finite,  we recall that the objective function:
$
-\lambda\varepsilon + \frac{1}{N} \sum_{j=1}^N s_j.
$
Since $\lambda \geq 0$ and $\varepsilon \geq 0$, the term $-\lambda\varepsilon \leq 0$.
Suppose, for contradiction, that the objective function is unbounded above. This would require $\frac{1}{N} \sum_{j=1}^N s_j \to \infty$, but this is impossible because $s_j$ are bounded due to the boundedness of $\mathfrak{X}$ and $z_j$.
Similarly, the objective function cannot be unbounded below (i.e., tending to $-\infty$) unless $\lambda \to \infty$ and $\varepsilon > 0$. However, increasing $\lambda$ without bound would decrease the objective value due to the $-\lambda\varepsilon$ term, which cannot be optimal when maximizing. Therefore, the dual optimal value is finite.

We now establish the proof that the optimal solution set is nonempty.
Note that we have established that the dual feasible set is nonempty. According to Theorem~\ref{theorem:convex reduction with non-zero transaction costs during rebalancing}, the dual problem is a convex optimization problem, and the feasible set is convex and closed.
Because the objective function is continuous and the feasible set is nonempty and closed, the maximum is achieved within the feasible set. Therefore, there exists at least one optimal solution~$(\lambda^*, s_j^*, z_j^*)$, meaning the set of dual optimal solutions is nonempty.

To complete the proof, it remains to show that the set of dual optimal solutions is bounded.
We need to demonstrate that the variables $\lambda$, $s_j$, and $z_j$ are bounded in the set of optimal solutions.

\emph{Boundedness of $\lambda$:} If $\lambda \to \infty$ and $\varepsilon > 0$, the objective $-\lambda\varepsilon + \frac{1}{N} \sum_{j=1}^N s_j$ tends to $-\infty$, which is worse than any finite objective value. Therefore, to maximize the objective, $\lambda$ must be bounded above.

\emph{Boundedness of $z_j$:} Since $\|z_j\|_* \leq \lambda$ and $\lambda$ is bounded, each $z_j$ is also bounded.

\emph{Boundedness of $s_j$:} From the constraint:
\[
\min_{\mathcal{X} \in \mathfrak{X}} \left[ \log\left( c(w) + w^\top \mathcal{X} \right) + z_j^\top (\mathcal{X} - \widehat{\mathcal{X}}_j) \right] \geq s_j,
\]
and the boundedness of $\mathfrak{X}$ and $z_j$, it follows that $s_j$ are bounded below. Similarly, $s_j$ are bounded above because the expression inside the minimum is bounded above (since $\mathcal{X}$ and $z_j$ are bounded).
Therefore, all variables in the dual problem are bounded in the set of optimal~solutions. \qedhere
\end{proof}

\end{document}